\documentclass[a4paper,reqno]{amsart}%
\usepackage{pgf,tikz}
\usepackage{amssymb}
\usepackage{amsmath}
\usepackage{amsfonts}
\usepackage{graphicx}
\usepackage{bm}%
\providecommand{\U}[1]{\protect\rule{.1in}{.1in}}
\usetikzlibrary{backgrounds}
\usetikzlibrary{arrows}
\let\oldmathbf\mathbf
\renewcommand{\mathbf}[1]{\boldsymbol{\oldmathbf{#1}}}
\newtheorem{theorem}{Theorem}

\newtheorem{corollary}[theorem]{Corollary}

\newtheorem{definition}[theorem]{Definition}

\newtheorem{lemma}[theorem]{Lemma}

\newtheorem{proposition}[theorem]{Proposition}
\newtheorem{remark}[theorem]{Remark}

\allowdisplaybreaks
\begin{document}
\title{Fourier analytic techniques for lattice point discrepancy}
\author[L. Brandolini]{Luca Brandolini}
\address{Dipartimento di Ingegneria Gestionale, dell'Informazione e della Produzione,
Universit\`a degli Studi di Bergamo, Viale Marconi 5, Dalmine BG, Italy}
\email{luca.brandolini@unibg.it}
\author[G. Travaglini]{Giancarlo Travaglini}
\address{Dipartimento 
di Matematica e Applicazioni, Universit\`a di Milano-Bicocca, Via
Cozzi 55, Milano, Italy}
\email{giancarlo.travaglini@unimib.it}

\date{}
\begin{abstract}
Counting integer points in large convex bodies with smooth boundaries containing
isolated flat points is oftentimes an intermediate case between balls (or convex bodies with smooth boundaries having everywhere positive
curvature) and cubes (or convex polytopes). In this paper we
provide a detailed description of several discrepancy problems in the
particular planar case where the boundary coincides locally with the graph of
the function $\mathbb{R\ni}t\mapsto\left\vert t\right\vert ^{\gamma}$, with
$\gamma>2$. We consider both \textit{integer points} problems and
\textit{irregularities of distribution} problems. The above \textquotedblleft
restriction\textquotedblright\ to a particular family of convex bodies is
compensated by the fact that many proofs are elementary. The paper
is entirely self-contained.
\end{abstract}
\maketitle

\section{Introduction}

The word \textit{discrepancy} comes from its Latin counterpart \textit{discrepantia
}(disagreement, contrast) and here expresses the deviation of a
\textit{discrete volume} of a convex body from its \textit{(continuous)
volume}. Much of this paper is devoted to the study of lattice points
discrepancy in dimension two: for a given convex body $C\subset\mathbb{R}^{2}$ (that is a compact convex set with non-empty interior) and a large real positive parameter
$R$ we compare the number of points with integer coordinates contained in the
dilated body
\[
RC=\left\{  t\in\mathbb{R}^{2}:t/R\in C\right\}
\]
and its area. More precisely we consider the discrepancy%
\[
\mathcal{D}\left(  RC\right)  :=-R^{2}\left\vert C\right\vert +\mathrm{card}%
\left(  RC\cap\mathbb{Z}^{2}\right)  =-R^{2}\left\vert C\right\vert
+\sum_{n\in\mathbb{Z}^{2}}\chi_{RC}\left(  n\right)
\]
where $\chi_{A}$ denotes the characteristic (indicator) function of the set
$A$.

The problem of estimating $\mathcal{D}\left(  RC\right)  $ for large values of
$R$ has a long history and several connections to different branches of
mathematics (see e.g. \cite{BC,BGT,Cha,DT,Hux,Kra,Mat,Tra}).

Here we are interested in the following specific family of convex bodies.

\begin{definition}
\label{definition}Let $\mathbb{R\ni\gamma}>2$. We denote by $C_{\gamma}$ any
planar compact convex set, contained in the square $\left(  -1/2,1/2\right)
^{2}$, whose boundary $\partial C_{\gamma}$ coincides, in a small neighbourhood
$U$ of the origin, with the graph of the function $\mathbb{R\ni}%
x\mapsto\left\vert x\right\vert ^{\gamma}$. We also assume that, outside
$\frac{1}{2}U$, $\partial C_{\gamma}$ is smooth with curvature $\geqslant c>0$.
\end{definition}
\begin{center}%
\[
\includegraphics[width=5cm]{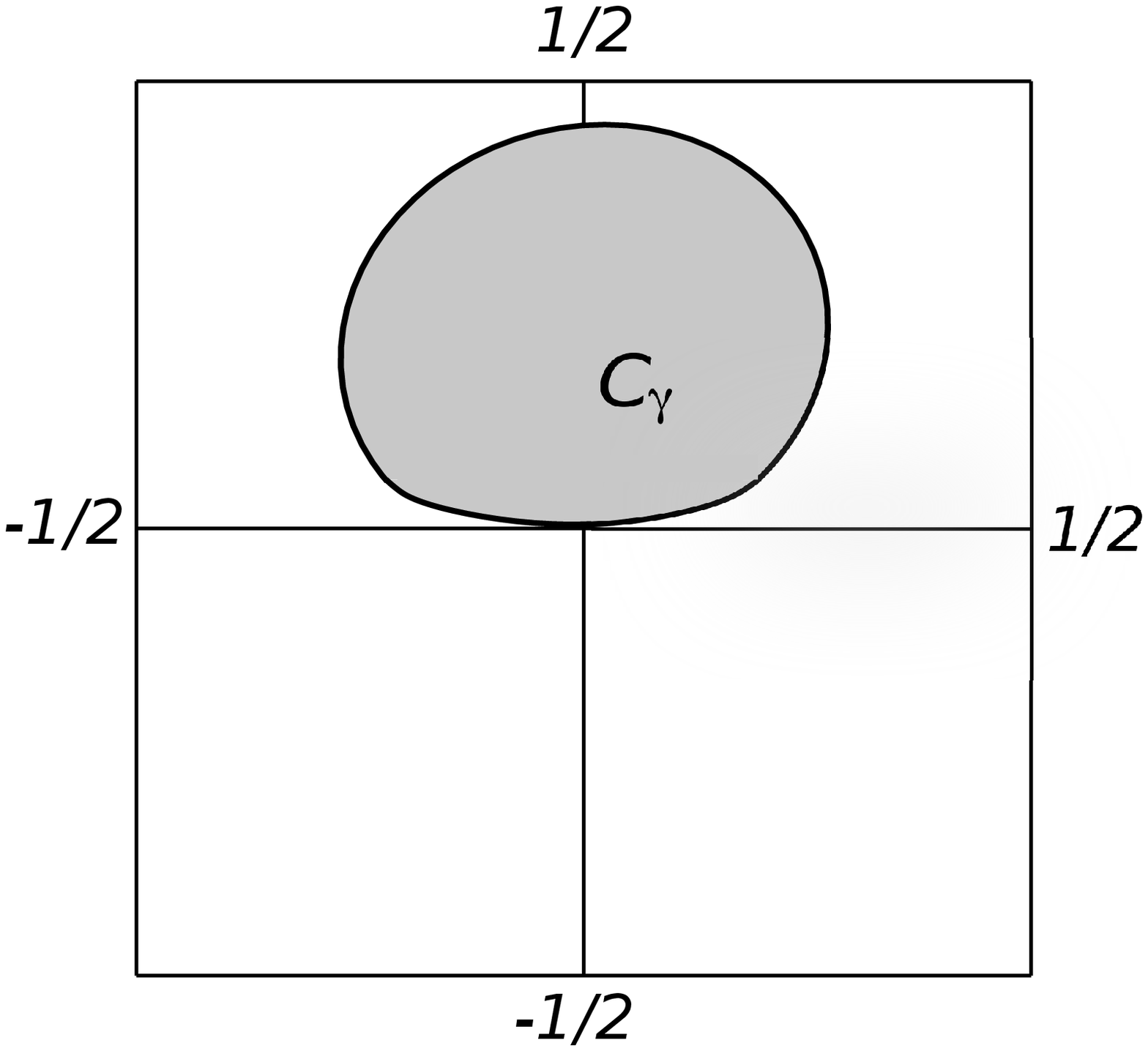}
\]
\end{center}
Our interest in the above class of convex bodies comes from the fact that a
large part of \textit{Geometric discrepancy} has been developed for rectangles
(or parallelepipeds, or polytopes) and discs (or balls, or convex bodies
having smooth boundary with everywhere positive Gaussian curvature). 
See the
above list of references and also \cite{Bec,Dav,Mon,Rot,Sch}.
The above index $\gamma$
provides a sort of \textquotedblleft bridge\textquotedblright\ between, say, a
disc and a square, which respectively can be roughly seen as the cases
$\gamma=2$ and $\gamma=\infty$. Anyway in the last section we shall see a
situation where $C_{\gamma}$ does not have this \textit{intermediate position},
and a sort of dichotomy appears.

\bigskip
The proofs in this paper are essentially Fourier analytic and several
arguments come from \cite{BCGGT}, \cite{BGT}, \cite{BRT} and \cite{Gar}. All
the results in this paper are essentially known, except Theorem
\ref{irredistr}.

\bigskip

We set the notation. 

We identify the torus $\mathbb{T}^{2}=\mathbb{R}%
^{2}/\mathbb{Z}^{2}$ with the unit square $\left[  -1/2,1/2\right)  ^{2}$. Let
$f\in L^{1}\left(  \mathbb{T}^{2}\right)  $ and for every $k\in\mathbb{Z}^{2}$
let%
\[
\widehat{f}\left(  k\right)  =\int_{\mathbb{T}^{2}}f\left(  t\right)  e^{-2\pi
it\cdot k}\ dt
\]
be the Fourier coefficient of $f\left(  t\right)  $, which therefore
has Fourier series%
\[
\sum_{k\in\mathbb{Z}^{2}}\widehat{f}\left(  k\right)  e^{2\pi it\cdot k}\ .
\]
The points in $\mathbb{Z}^{2}$ are termed \textit{integer points. }If $g\in
L^{1}\left(  \mathbb{R}^{2}\right)  $ and $\xi\in\mathbb{R}^{2}$ then
\[
\widehat{g}\left(  \xi\right)  =\int_{\mathbb{R}^{2}}g\left(  t\right)
e^{-2\pi it\cdot\xi}\ dt
\]
denotes the Fourier transform of $g\left(  t\right)  $.

The connection between the above discrepancy and Fourier analysis is a
consequence of the following simple observation. Let $C$ be a convex body in
$\mathbb{R}^{2}$ and, for every $t\in\mathbb{R}^{2}$, define the
\textit{discrepancy function}%
\begin{align*}
\mathcal{D}_{R}\left(  t\right)   &  =\mathcal{D}\left(  RC+t\right)
=-R^{2}\left\vert C\right\vert +\mathrm{card}\left(  (RC+t)\cap\mathbb{Z}%
^{2}\right)  \\
&  =-R^{2}\left\vert C\right\vert +\sum_{n\in\mathbb{Z}^{2}}\chi_{RC}\left(
n-t\right)  \ .
\end{align*}
The function $\mathcal{D}_{R}\left(  t\right)  $ is periodic  with Fourier series%
\begin{equation}
\sum_{0\neq m\in\mathbb{Z}^{2}}\widehat{\mathcal{D}}_{R}\left(  m\right)
e^{2\pi im\cdot t}=\sum_{0\neq m\in\mathbb{Z}^{2}}\widehat{\chi}_{RC}\left(
m\right)  e^{2\pi im\cdot t}\ .\label{fourierdiscr}%
\end{equation}
Indeed,%
\begin{align*}
\widehat{\mathcal{D}}_{R}\left(  0\right)   &  =\int_{\mathbb{T}^{2}}\left(
-R^{2}\left\vert C\right\vert +\sum_{n\in\mathbb{Z}^{2}}\chi_{RC}\left(
n-t\right)  \right)  \ dt\\
&  =-R^{2}\left\vert C\right\vert +\sum_{n\in\mathbb{Z}^{2}}\int
_{\mathbb{T}^{2}}\chi_{RC}\left(  n-t\right)  \ dt
=-R^{2}\left\vert
C\right\vert +\int_{\mathbb{R}^{2}}\chi_{RC}\left(  t\right)  \ dt=0\ ,
\end{align*}
and for $m\neq0$,
\begin{align*}
\widehat{\mathcal{D}}_{R}\left(  m\right)   &  =\int_{\mathbb{T}^{2}}\left(
-R^{d}\left\vert C\right\vert +\sum_{n\in\mathbb{Z}^{2}}\chi_{RC}\left(
n-t\right)  \right)  e^{-2\pi im\cdot t}\ dt\\
&  =\int_{RC}e^{-2\pi iRm\cdot t}\ dt=\widehat{\chi}_{RC}\left(  m\right)  \ .
\end{align*}
Observe that the two sides of the equality $\widehat{\mathcal{D}}_{R}\left(
m\right)  =\widehat{\chi}_{RC}\left(  m\right)  $ have a different nature. On
the LHS the terms $\widehat{\mathcal{D}}_{R}\left(  m\right)  $ are the
Fourier coefficients of the periodic function $\mathcal{D}_{R}\left(  t\right)
$ (defined on $\mathbb{T}^{2}$), while on the RHS the terms $\widehat{\chi
}_{RC}\left(  m\right)  $ are the restriction (to $\mathbb{Z}^{2}$) of the
Fourier transform $\widehat{\chi}_{RC}\left(  \xi\right)  $ of the function
$\chi_{RC}\left(  t\right)  $ (which is defined on $\mathbb{R}^{2}$).

Throughout the paper $c$, $c_{1}$, $c_{2}$, $\ldots$ denote constants which
may change from step to step.

\section{Integer points in large convex bodies}

First we recall the \textit{circle problem} and the \textit{Hardy-Voronoi
identity}. Let $R$ be a positive real number. The circle problem asks for a significant estimate of the sum
\[
A\left(  R\right)  =\sum_{0\leq k\leq R^2}r\left(  k\right)
\]
of the arithmetic function%
\[
r\left(  k\right)  =\mathrm{card}\left\{  \left(  m_{1},m_{2}\right)
\in\mathbb{Z}^{2}:m_{1}^{2}+m_{2}^{2}=k\right\}  \ ,
\]
that is the number of ways of writing a non-negative  integer as a sum of two squares.
Let
$B=B\left(  0,1\right)  =\left\{  t\in\mathbb{R}^{2}:\left\vert t\right\vert
\leqslant1\right\}  $ be the disc of unit radius centred at the origin. More generally
we write $B\left(  \tau,r\right)  :=\left\{  t\in\mathbb{R}^{2}:\left\vert
t-\tau\right\vert \leqslant r\right\}  $. 

More than two hundreds years ago
C.F. Gauss observed that the average of $r\left(  k\right)  $ reduces to
counting the integer points in the dilated disc 
$RB=\left\{  t\in \mathbb{R}^{2}:\left\vert t/R\right\vert \leqslant1\right\}  $, for $R>1$.
Then it is easy to observe that $\mathrm{card}\left(  RB\cap\mathbb{Z}%
^{2}\right)  $ equals the area $R^{2}\pi$ of the disc plus an error term
smaller, in absolute value, than ($\sqrt{2}$ times) the length of the boundary
of the dilated disc. That is
\[
\mathrm{card}\left(  RB\cap\mathbb{Z}^{2}\right)  =R^{2}\pi+\mathcal{D}\left(
RB\right)  \ ,
\]
with $\mathcal{D}\left(  RB\right)  =\mathcal{O}\left(  R\right)  $.
The error bound $\mathcal{O}\left(  R\right)  $ has been improved several
times during the last century. In 1906 W. Sierpi\'{n}ski proved
that$\ \ \left\vert \mathcal{D}\left(  RB\right)  \right\vert \leqslant
cR^{2/3}$. The best result so far ($\leqslant cR^{0.627\cdots}$) has been recently obtained by J. Bourgain and N. Watt \cite{BW}.

In 1916 G. Hardy proved that the exponent $1/2$ is not large enough and
conjectured that $\ \left\vert \mathcal{D}\left(  RB\right)  \right\vert
\leqslant cR^{1/2+\varepsilon}\ $.

Earlier in 1915 G. Hardy proved the following result (previously conjectured
by G. Voronoi):
\begin{equation}
R\sum_{k=1}^{+\infty}\frac{r\left(  k\right)  }{\sqrt{k}}J_{1}\left(
2\pi\sqrt{k}R\right)  =\frac{A\left(  R^{+}\right)  +A\left(  R^{-}\right)
}{2}-\pi R^{2} \label{hardVor}\;,
\end{equation}
where $A\left(  R^{+}\right)$ and $A\left(  R^{-}\right)$ denote the right and left limits at $R$ respectively of the discontinuous function $A(x)$, and
\[
J_{1}\left(  x\right)  =\frac{x}{2}\int_{-1}^{1}\left(  1-t^{2}%
\right)  ^{1/2}e^{itx}\ dt
\] 
is a Bessel function, thereby giving an
analytic expression for the discrepancy. See \cite{BCol}, \cite{IK}.

The series in (\ref{hardVor}) is the spherical Fourier series (see
(\ref{fourierdiscr}))
\[
\sum_{0\neq m\in\mathbb{Z}^{2}}\widehat{\chi}_{RB}\left(  m\right)  e^{2\pi
im\cdot t}=\lim_{K\rightarrow+\infty}\sum_{0<\left\vert m\right\vert
\leqslant K}\widehat{\chi}_{RB}\left(  m\right)  e^{2\pi im\cdot t}%
\]
of the discrepancy function $\mathbb{T}^{2}\ni t\longmapsto\mathcal{D}\left(
RB+t\right)  $, evaluated at the origin. Indeed, for every $0\neq\xi
\in\mathbb{R}^{2}$, we have%
\[
\widehat{\chi}_{B}\left(  \xi\right)  =\left\vert \xi\right\vert ^{-1}%
J_{1}\left(  2\pi\left\vert \xi\right\vert \right)
\]
(see e.g. \cite[p.216]{Tra}) and therefore, after summing on the integers
points $m$ on all circles of radius $\sqrt{k}$, we obtain, at $t=0$,%
\begin{align}
\sum_{m\neq0}\widehat{\chi}_{RB}\left(  m\right)   &  =R^{2}\sum_{m\neq
0}\widehat{\chi}_{B}\left(  Rm\right)  =R\sum_{m\neq0}\left\vert m\right\vert
^{-1}J_{1}\left(  2\pi R\left\vert m\right\vert \right)  \label{Kendall}\\
&  =R\sum_{j=1}^{+\infty}\frac{r\left(  k\right)  }{\sqrt{k}}J_{1}\left(  2\pi
R\sqrt{k}\right)  \ .\nonumber
\end{align}
The above series is not absolutely convergent and, in spite of its explicit
expression, does not seem to help us in funding a sharp bound for the
discrepancy, unless we apply a smoothing argument of E. Hlawka which turns the
above series into an absolutely convergent one, and provides a new proof of
Sierpi\'{n}ski's estimate (see e.g. \cite[p. 162]{Tra} or the proof of Theorem
\ref{minore dpiu1} below).

More generally, when $C$ is a convex planar body, the discrepancy function%
\[
\mathcal{D}_{R}\left(  t\right)  =-R^{2}\left\vert C\right\vert +\mathrm{card}%
\left(  (RC+t)\cap\mathbb{Z}^{2}\right)
\]
is a periodic piecewise constant function (observe that $\mathcal{D}%
_{R}\left(  t\right)  $ may change value only when, moving $t$, we hit or we leave integer points). The above Hardy-Voronoi identity falls within the framework
of pointwise convergence of Fourier series of piecewise smooth functions. A
simple nice result in this field says that if the graph of $f(t)$ has the shape
in the following  figure, about a point $t_{0}$, then the spherical means of the above
Fourier series converge, at the point $t_{0}$, to the number $b\beta/2\pi$%

\[
\includegraphics[width=7cm]{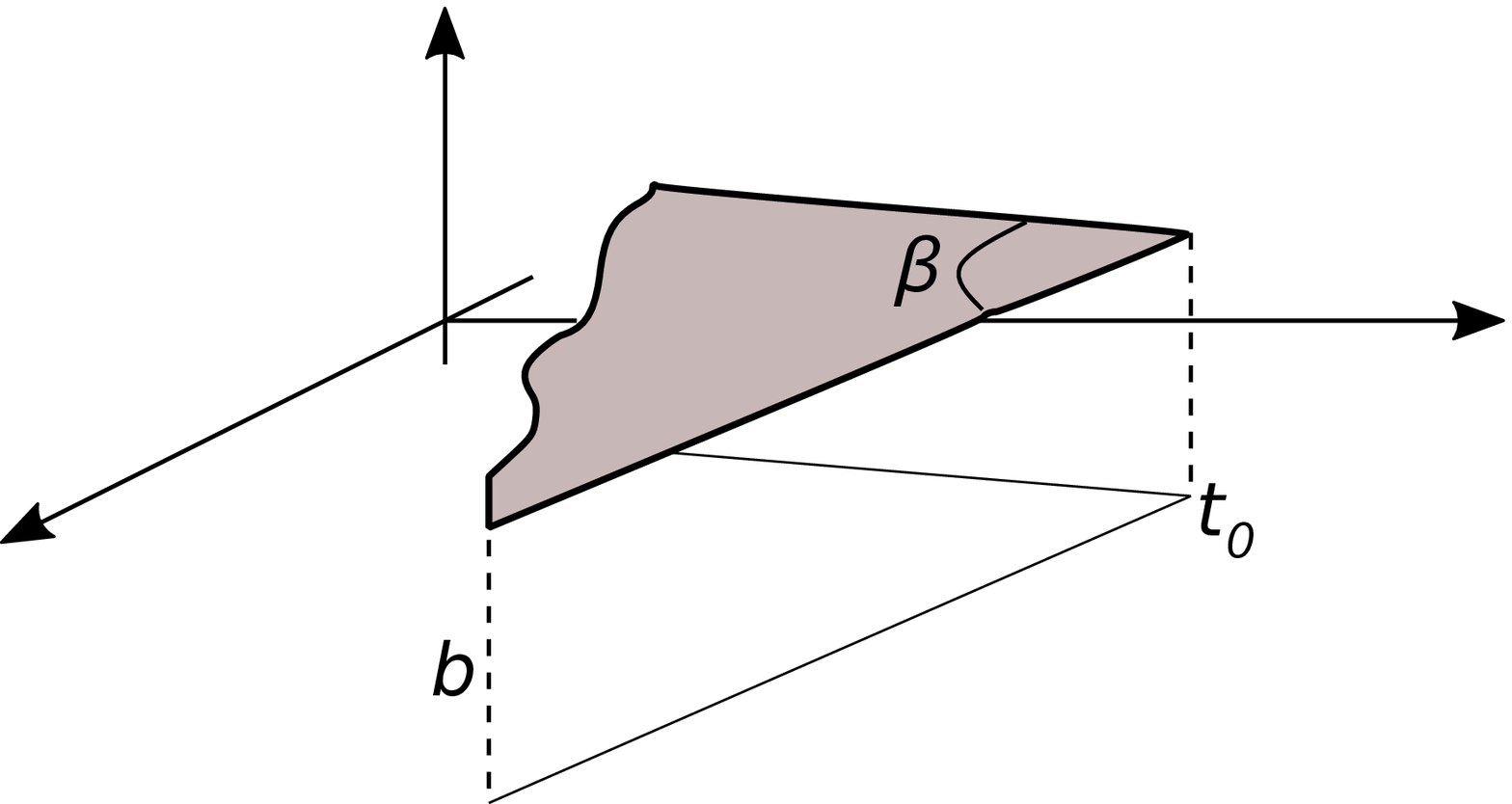}
\]
(see e.g. \cite{BCol}).

The situation may improve if we introduce an $L^{2}$ average (over
translations) of the discrepancy function $\mathcal{D}_{R}\left(  t\right)  $.

\subsection{Kendall's argument}

D. Kendall \cite{Ken} was the first one to write explicitly the Fourier
series of the discrepancy function (and therefore to point out the identity
(\ref{Kendall})). Then he used the Parseval identity to prove that for, say, the
unit disc $B$ we have%
\[
\left\{  \int_{\mathbb{T}^{2}}\left\vert \mathcal{D}\left(  RB+t\right)
\right\vert ^{2}\ dt\right\}  ^{1/2}\leqslant cR^{1/2}\ .
\]
Indeed it is known (by the asymptotics of Bessel functions or by Theorem
\ref{chord} below) that
\[
\left\vert \widehat{\chi}_{B}\left(  \xi\right)  \right\vert \leqslant
c\left(  1+\left\vert \xi\right\vert \right)  ^{-3/2}\;.
\]
Therefore%
\begin{equation}
\int_{\mathbb{T}^{2}}\left\vert \mathcal{D}\left(  RB+t\right)  \right\vert
^{2}\ dt=R^{4}\sum_{m\neq0}\left\vert \widehat{\chi}_{B}\left(  Rm\right)
\right\vert ^{2}\leqslant cR\sum_{m\neq0}\left\vert m\right\vert
^{-3}=cR\ .\label{kendall}%
\end{equation}

Kendall's result for the disc can be extended to the case of an arbitrary planar
convex body $C$ as long as we introduce an average over rotations. A.
Podkorytov (see \cite{Pod91}, see also \cite[p.176]{Tra},\cite{BHI}) proved
that for every planar convex body $C$ we have%
\[
\int_{0}^{2\pi}\left\vert \widehat{\chi}_{C}\left(  \rho\Theta\right)
\right\vert ^{2}d\theta\leqslant c\ \rho^{-3}\ ,
\]
where $\Theta=\left(\cos\theta,\sin\theta\right)$ and $\rho\geqslant 2$.
This and Kendall's argument yield%
\begin{equation}
\left\{  \int_{SO\left(  2\right)  }\int_{\mathbb{T}^{2}}\left\vert
\mathcal{D}\left(  \sigma\left(  RC\right)  +t\right)  \right\vert
^{2}\ dtd\sigma\right\}  ^{1/2}\leqslant cR^{1/2}\label{kendallPodkorytov}%
\end{equation}
for every planar convex body $C$. Note that, within the family of convex
planar bodies having piecewise smooth boundary, the upper bound
(\ref{kendallPodkorytov}) can be inverted (see \cite{TT}, \cite{BRT}) if and only if $C$
is not a polygon that is symmetric and can be inscribed in a circle.

\bigskip

Kendall's $L^{2}$ result for the disc can be extended to $L^{p}$ spaces
provided $p<4$ (see \cite{Hux1}, \cite{BCGT}).

\begin{theorem}
\label{twoandp} Let $B$ be the unit disc. Then%
\begin{equation}
\left\{  \int_{\mathbb{T}^{2}}\left\vert \mathcal{D}\left(  RB+t\right)
\right\vert ^{p}\ dt\right\}  ^{1/p}\leqslant c\left\{
\begin{array}
[c]{lll}%
R^{1/2} &  & \text{if }1\leqslant p<4,\\
R^{1/2}\log^{1/4}\left(  R\right)   &  & \text{if }p=4,\\
R^{2/3\left(  1-1/p\right)  } &  & \text{if }p>4.
\end{array}
\right.  \label{cambridge proc}%
\end{equation}

\end{theorem}

The idea for the proof of (\ref{cambridge proc}) is that in Kendall's argument
the series $\ \sum_{m\neq0}\left\vert m\right\vert ^{-3}\ $ converges
\textquotedblleft more than enough\textquotedblright\ and we have room for a
few positive results when $p>2$. Actually the upper bounds in Theorem
\ref{twoandp} are known to be sharp in the range $1\leqslant p<4$. The case
$p\geqslant4$ uses Hlawka's smoothing argument and it does not seem to be sharp.

\subsection{Integer points in large polygons}

The study of integer points in polyhedra is another topic with several
applications in different parts of mathematics (see e.g. \cite{Bar},\cite{BR},\cite{Schr}).
\begin{center}%
\[
\includegraphics[scale=0.32]{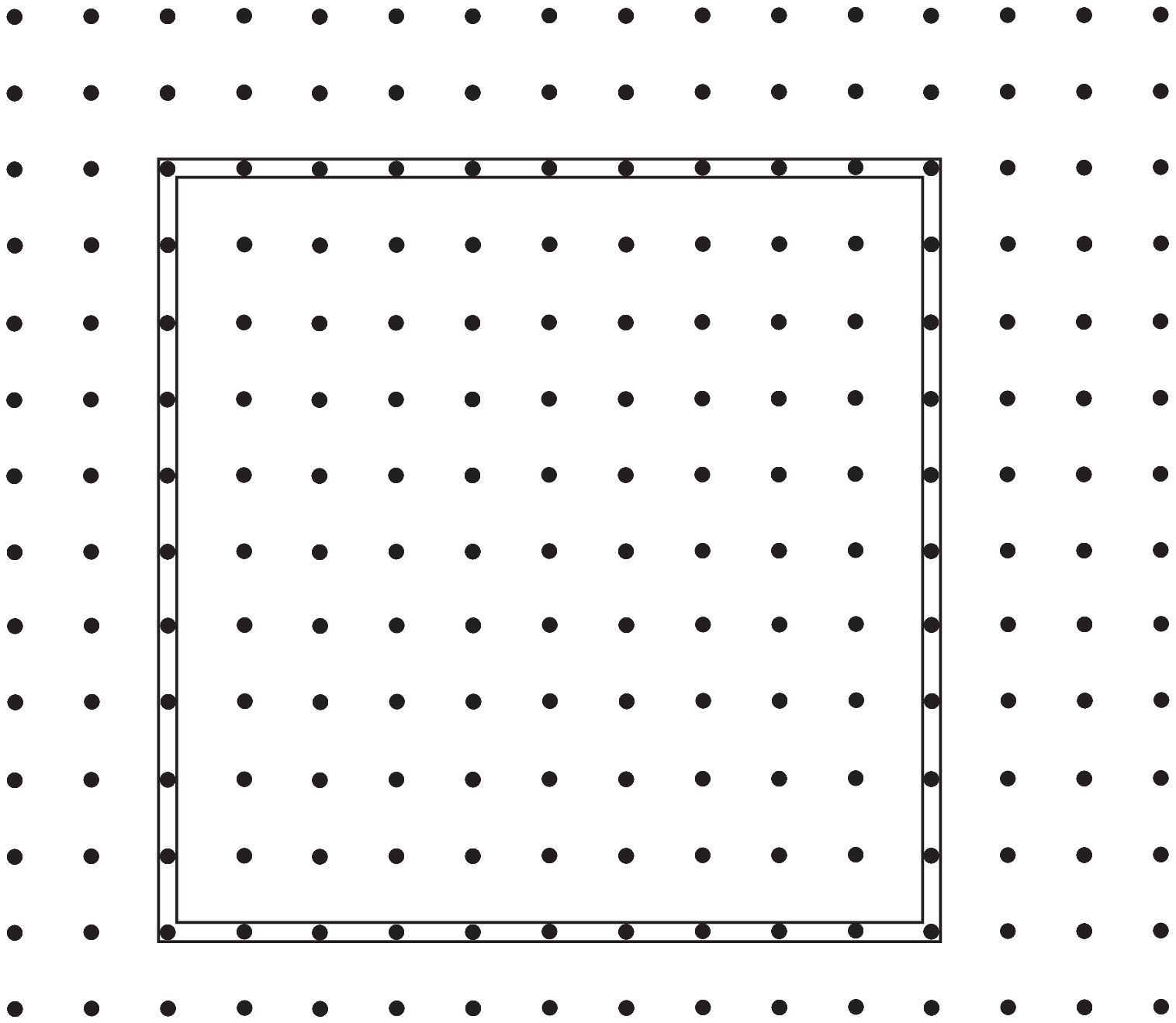}
\]
\end{center}
As a first (trivial) example we consider a square having sides parallel to the
axes. Then it is easy to check that the discrepancy is $\approx R$ for
infinitely many large values of $R$. Indeed we see that  the two squares of side $\approx R$ in the previous figure have essentially the same
area, but one has $\approx R$ integer points more than the other.

A suitable rotation of the square may make the discrepancy for the square very
small. H. Davenport (see \cite{Dav}) has proved that if a square $Q$ has slope
(say) $\sqrt{2}$ then%
\[
\int_{\mathbb{T}^{2}}\left\vert \mathcal{D}\left(  RQ+t\right)  \right\vert
^{2}\ dt\leqslant c\ \log\left(  R\right)  \ .
\]

A logarithmic estimate holds true also after averaging over rotations. In
\cite{BCT} it is proved that the discrepancy associated to a polygon $P$
satisfies, for $R \geqslant 2$,
\begin{equation}
\int_{SO\left(  2\right)  }\left\vert \mathcal{D}\left(  R\sigma\left(
P\right)  \right)  \right\vert \ d\sigma\leqslant c\ \log^{2}\left(  R\right)
\ .\label{L uno discrepanza}%
\end{equation}
Moreover this estimate is almost sharp in the following sense. For a triangle
$S\subset\mathbb{R}^{2}$ we have%
\[
\int_{\mathbb{T}^{2}}\int_{SO\left(  2\right)  }\left\vert \mathcal{D}\left(
R\sigma\left(  S\right)  +t\right)  \right\vert \ d\sigma dt\geqslant
c\ \log\left(  R\right)  \ .
\]

\section{Pointwise estimates for $\widehat{\chi}_{C_{\gamma}}\left(
\xi\right)  $}

To study the discrepancy for $C_{\gamma}$ we need careful estimates 
of the Fourier transform of the function
$\chi_{C_{\gamma}}\left(  t\right)  $. We start with a general result, see \cite{Pod91} and also \cite{BNW}
for a result in higher dimension.
\begin{theorem}
\label{chord}Let $C\subset\mathbb{R}^{2}$ be a strictly convex body with
piecewise smooth boundary. We write $\Theta=(\cos\theta,\sin\theta)$ and, for
$0\leqslant\theta<2\pi$ and small $\delta>0$, let%
\[
\lambda(\delta,\theta)=\left\{  t\in C:\delta+t\cdot\Theta=\sup\limits_{y\in
C}\left(  y\cdot\Theta\right)  \right\}
\]
be the chord perpendicular to $\Theta$ \textquotedblleft at distance
$\delta$ from the boundary\textquotedblright\ $\partial C$ of $C$ (see the
following figure). Then, there exist $c_{1}$ and $c_{2}$ independent of
$\theta$ such that, for $\rho>c_{1}$, we have%
\[
\left\vert \widehat{\chi}_{C}(\rho\Theta)\right\vert \leqslant c_{2}%
\,\rho^{-1}\left(  \left\vert \lambda(\rho^{-1},\theta)\right\vert +\left\vert
\lambda(\rho^{-1},\theta+\pi)\right\vert \right)  \;,
\]
where $\left\vert \lambda\right\vert $ denotes the length of the segment
$\lambda$.
\end{theorem}
\begin{center}%
\[
\includegraphics[scale=0.45]{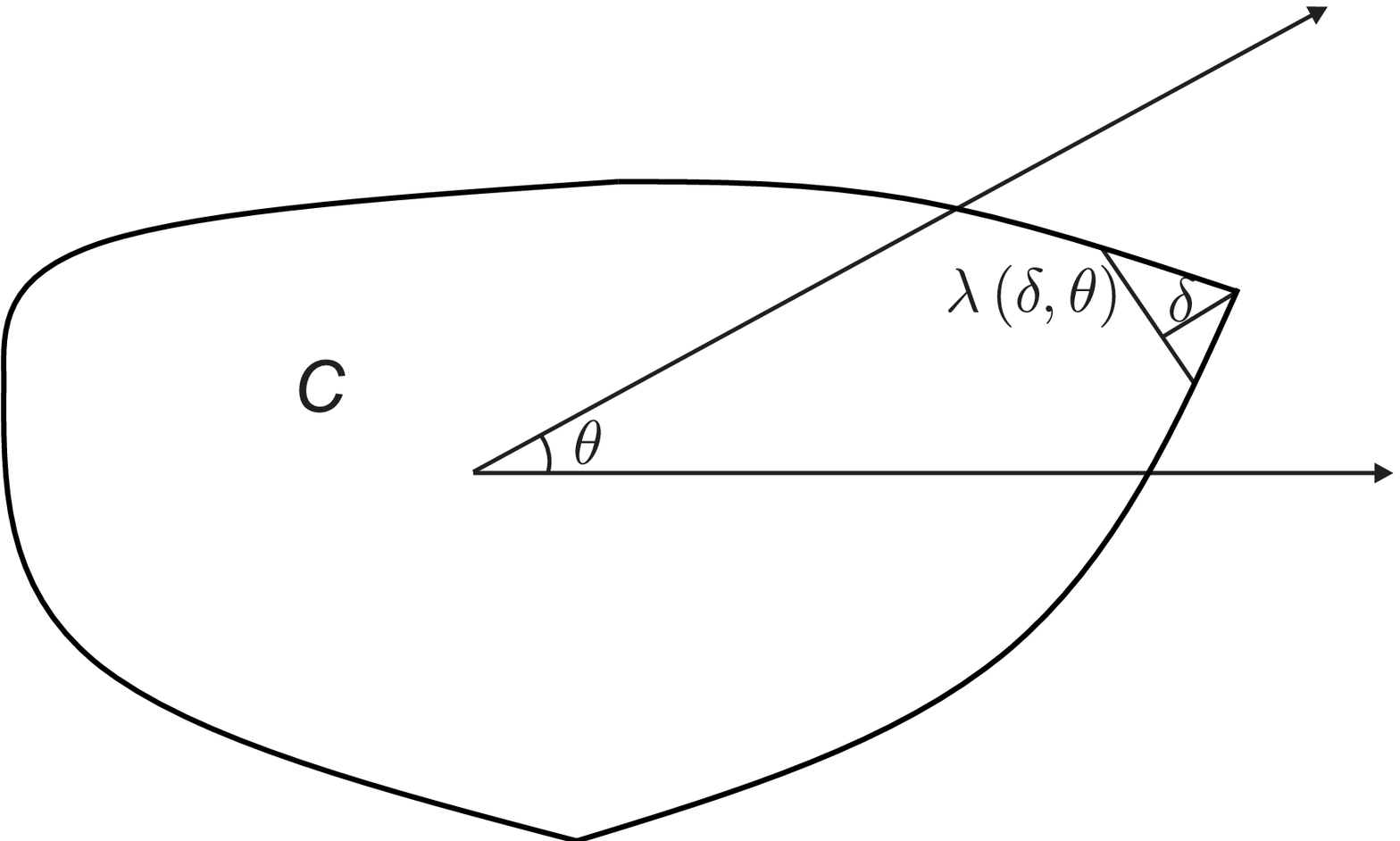}
\]
\end{center}

\medskip
\begin{proof}
We may assume $\Theta=(1,0)$, so that we consider%
\[
\widehat{\chi}_{C}(\xi,0)=\int_{-\infty}^{+\infty}\left(  \int_{-\infty
}^{+\infty}\chi_{C}(t_{1},t_{2})\ dt_{2}\right)  \ e^{-2\pi i\xi t_{1}}%
~dt_{1}=\widehat{h}(\xi)\;, \label{radon}
\]
where $h\left(  x\right)  $ is the length of the segment given by the
intersection of $C$ with the line $t_{1}=x$ (we can say that the $2$-dimensional Fourier transform is a $1$-dimensional Fourier transform of a Radon transform). Observe that the function
$h\left(  x\right)  $ is continuous on $\mathbb{R}$ and strictly concave on
its support, which we may assume to be the interval $\left[  -1,1\right]  $.
We may assume that $h\left(  x\right)  $ attains its maximum at some
$\beta\geqslant0$ (the other case being similar).%

\[
\includegraphics[scale=0.75]{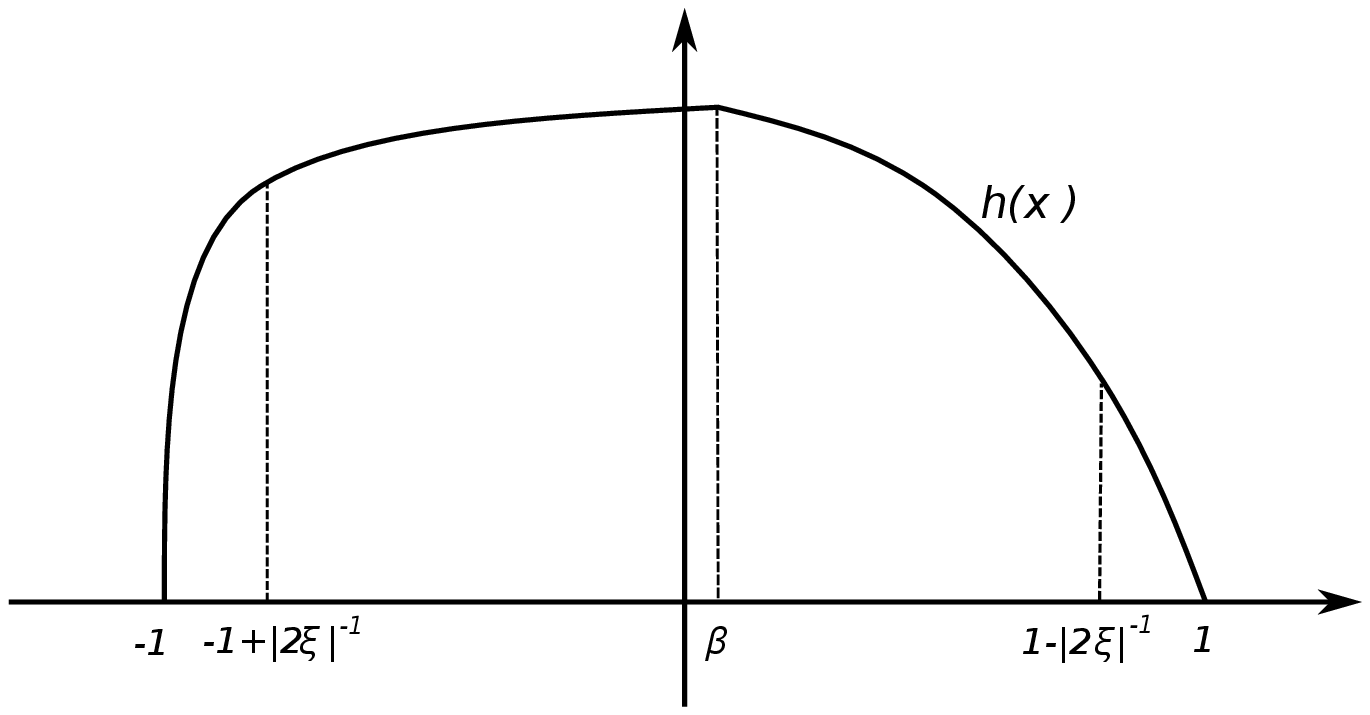}
\]
The strict convexity implies the continuity of $h\left(  x\right)  $, so that
$h\left(  -1\right)  =h\left(  1\right)  =0$. We may assume $\xi>1$. Then
integration by parts yields
\begin{align*}
\widehat{h}(\xi) &  =\int_{-1}^{1}h\left(  x\right)  e^{-2\pi i\xi
x}\ dx=\frac{1}{2\pi i\xi}\int_{-1}^{1}h^{\prime}\left(  x\right)  e^{-2\pi
i\xi x}\ dx\\
&  =\frac{-1}{2\pi i\xi}\int_{-1+\left(  2\xi\right)  ^{-1}}^{1+\left(
2\xi\right)  ^{-1}}h^{\prime}\left(  x-\frac{1}{2\xi}\right)  e^{-2\pi i\xi
x}\ dx\ .
\end{align*}
Hence%
\begin{align*}
2\left(  2\pi i\xi\right)  \widehat{h}(\xi) &  =\int_{-1}^{-1+\left(
2\xi\right)  ^{-1}}h^{\prime}\left(  x\right)  e^{-2\pi i\xi x}\ dx\\
&  +\int_{-1+\left(  2\xi\right)  ^{-1}}^{1}\left(  h^{\prime}\left(
x\right)  -h^{\prime}\left(  x-\frac{1}{2\xi}\right)  \right)  e^{-2\pi i\xi
x}\ dx\\
&  +\int_{1}^{1+\left(  2\xi\right)  ^{-1}}h^{\prime}\left(  x-\frac{1}{2\xi
}\right)  e^{-2\pi i\xi x}\ dx\\
&  =I_{1}+I_{2}+I_{3}\ ,
\end{align*}
say.  Since $h\left(  x\right)  $ is increasing on $-1\leqslant x\leqslant0$ we
have
\[
\left\vert I_{1}\right\vert \leqslant\int_{-1}^{-1+\left(  2\xi\right)  ^{-1}%
}\left\vert h^{\prime}\left(  x\right)  \right\vert \ dx=\int_{-1}^{-1+\left(
2\xi\right)  ^{-1}}h^{\prime}\left(  x\right)  \ dx=h\left(  -1+\frac{1}{2\xi
}\right)  .
\]
In the same way, since $h^{\prime}\left(  x\right)  $ is decreasing, we have%
\begin{align*}
\left\vert I_{2}\right\vert  &  \leqslant-\int_{-1+\left(  2\xi\right)  ^{-1}%
}^{1}\left(  h^{\prime}\left(  x\right)  -h^{\prime}\left(  x-\frac{1}{2\xi
}\right)  \right)  \ dx\\
&  =h\left(  -1+\frac{1}{2\xi}\right)  +h\left(  1-\frac{1}{2\xi}\right)  \ .
\end{align*}
In order to estimate $I_{3}$ we consider two cases. Let $\beta\in\left[
0,1\right]  $ be the point where $h\left(  x\right)  $ attains its maximum. If
$\beta\leqslant1-\left(  2\xi\right)  ^{-1}$ we argue as we did for $I_{1}$.
If $1-\left(  2\xi\right)  ^{-1}\leqslant\beta<1$ we have
\begin{align*}
\left\vert I_{3}\right\vert  &  =\left\vert \int_{1}^{1+\left(  2\xi\right)
^{-1}}h^{\prime}\left(  x-\frac{1}{2\xi}\right)  e^{-2\pi i\xi x}%
\ dx\right\vert \\
&  \leqslant\left\vert \int_{1}^{\beta+\left(  2\xi\right)  ^{-1}}h^{\prime
}\left(  x-\frac{1}{2\xi}\right)  e^{-2\pi i\xi x}\ dx\right\vert 
  +\left\vert \int_{\beta+\left(  2\xi\right)  ^{-1}}^{1+\left(  2\xi\right)
^{-1}}h^{\prime}\left(  x-\frac{1}{2\xi}\right)  e^{-2\pi i\xi x}%
\ dx\right\vert \\
&  \leqslant\int_{1}^{\beta+\left(  2\xi\right)  ^{-1}}\left\vert h^{\prime
}\left(  x-\frac{1}{2\xi}\right)  \right\vert \ dx+\int_{\beta+\left(
2\xi\right)  ^{-1}}^{1+\left(  2\xi\right)  ^{-1}}\left\vert h^{\prime}\left(
x-\frac{1}{2\xi}\right)  \right\vert \ dx\\
&  \leqslant\int_{1}^{\beta+\left(  2\xi\right)  ^{-1}}h^{\prime}\left(
x-\frac{1}{2\xi}\right)  \ dx-\int_{\beta+\left(  2\xi\right)  ^{-1}%
}^{1+\left(  2\xi\right)  ^{-1}}h^{\prime}\left(  x-\frac{1}{2\xi}\right)
\ dx\\
&  =2h\left(  \beta\right)  -h\left(  1-\frac{1}{2\xi}\right)  \leqslant
4h\left(  0\right)  -h\left(  1-\frac{1}{2\xi}\right)  \leqslant3h\left(
1-\frac{1}{2\xi}\right)  \ ,
\end{align*}
by the concavity of $h\left(  x\right)  $. This completes the proof.
\end{proof}

\begin{corollary}
\label{postivecurvature}Let $C$ be a planar convex body having smooth boundary
with strictly positive curvature. Then, for every $\left\vert \xi\right\vert
\geqslant1$, we have
\begin{equation}
\left\vert \widehat{\chi}_{C}\left(  \xi\right)  \right\vert \leqslant
\kappa\ \left\vert \xi\right\vert ^{-3/2}\label{stimacurvpos}%
\end{equation}
(where $\kappa\ $depends on $C$) .
\end{corollary}

\begin{proof}
We choose a point in $\partial C$, which we may assume to be the origin. We
also assume that $C$ is contained in the right half-plane and that $C$
contains a ball of radius $1$. For the sake of simplicity, we may also assume
that $\partial C$ is locally (that is for $\left\vert y\right\vert \leqslant
c$) the graph of an even function $g\left(  y\right)  $ satisfying $g\left(
0\right)  =g^{\prime}\left(  0\right)  =0$ and $\left\vert g^{\prime}\left(
y\right)  \right\vert \leqslant c$. Hence we consider only $0\leqslant
y\leqslant c$, so that $2g\left(  y\right)  $ is the inverse of the function
$h\left(  x\right)  $ described at the beginning of the proof of Theorem
\ref{chord}. Moreover our assumptions imply that (see again Theorem
\ref{chord} for the notation) 
\[
h\left(  \delta\right)  =\frac{1}{2}\left\vert
\lambda\left(  \delta,-\pi\right)  \right\vert
\]
and $h\left(  \delta\right)
$ is strictly increasing for $0\leqslant\delta\leqslant1$. The curvature
$K\left(  y\right)  $ at the point $\left(  g\left(  y\right)  ,y\right)
\in\partial C$ satisfies $c_{1}\leqslant K\left(  y\right)  \leqslant c_{2}$
(where $c_{1}$ and $c_{2}$ depend on the convex body $C$). Since%
\[
g^{\prime\prime}\left(  y\right)  =\left(  1+\left[  g^{\prime}\left(
y\right)  \right]  ^{2}\right)  ^{3/2}K\left(  y\right)  \ ,
\]
we have %
\[
g\left(  y\right)  =\int_{0}^{y}\left(  y-t\right)  g^{\prime\prime}\left(
t\right)  \ dt\approx\int_{0}^{y}\left(  y-t\right)  \ dt\approx y^{2} \;,
\]%
where $A\approx B$ means that $A$ and $B$ are positive and, for suitable
constants $c_{1},c_{2}$, we have $c_{1}A\leqslant B\leqslant c_{2}A$.
\[
\includegraphics[width=5cm]{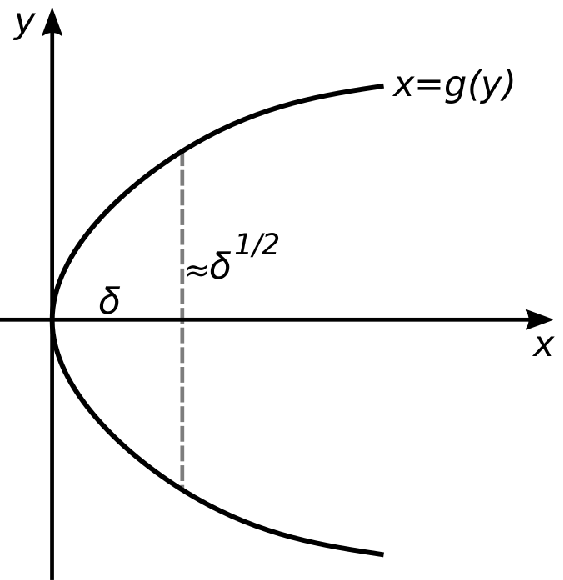}
\]
Then Theorem \ref{chord} yields%
\[
\left\vert \lambda\left(  \delta,-\pi\right)  \right\vert =h\left(
\delta\right)  \approx\delta^{1/2}%
\]
and therefore (\ref{stimacurvpos}).
\end{proof}

\begin{remark}
The estimate (\ref{stimacurvpos}) still holds under the less strict and more geometric
assumption that $C$ is a convex body that can roll unimpeded inside a disc.
See \cite{BCGT}. Observe that no convex polygon or convex body with smooth boundary having a flat point of order $> 2$ can roll unimpeded inside a disc.
\end{remark}

\begin{center}
\centering\begin{tikzpicture}[line cap=round,line join=round,>=stealth,scale=0.45]
\filldraw[fill=green!25,fill opacity=0.5] plot[shift={(11.5358983849,-1.)},line width=1.pt,domain=-0.523598775598:0.523598775598,variable=\t
]({1.*4.*cos(\t r)+0.*4.*sin(\t r)},{0.*4.*cos(\t r)+1.*4.*sin(\t r)}) --
plot[shift={(15.,-3.)},line width=1.pt,domain=1.57079632679:2.61799387799,variable=\t]({1.*4.*cos(\t
r)+0.*4.*sin(\t r)},{0.*4.*cos(\t r)+1.*4.*sin(\t r)}) --
plot[shift={(15.,1.)},line width=1.pt,domain=3.66519142919:4.71238898038,variable=\t]({1.*4.*cos(\t
r)+0.*4.*sin(\t r)},{0.*4.*cos(\t r)+1.*4.*sin(\t r)});
\draw[line width=1.pt] (11.,-3.) circle (5.65685424949cm);
\draw
[line width=0.5pt,dash pattern=on 3pt off 3pt] (11.2824898206,2.64979641237)-- (16.6519159808,-2.7636829555);
\draw[->,line width=1.pt] (5.,0.)-- (9.,4.);
\draw
[<->,line width=0.2pt,dash pattern=on 1pt off 1pt] (15,1) -- (13.96,-0.04);
\draw(12.662833788,-0.728652444091) node[anchor=north west] {$\Omega$};
\draw(8,-3) node[anchor=north west] {$\Delta$};
\draw(5.97500419166,3.09871991528) node[anchor=north west] {$\theta$};
\draw(14.1,.8) node[anchor=north west] {$\scriptsize{\delta}$};
\end{tikzpicture}


\end{center}

\begin{remark}
\label{pococurv}Assume that $C$ is a convex planar body with piecewise smooth
boundary. Without any assumptions on the curvature the estimate
(\ref{stimacurvpos}) may fail. However Theorem \ref{radon} and integration by parts
show that
\begin{equation}
\left\vert \widehat{\chi}_{C}\left(  \xi\right)  \right\vert \leqslant
c\left\vert \xi\right\vert ^{-1}\label{decay 1} \;,
\end{equation}
whenever $\left\vert \xi\right\vert \geqslant1$.
\end{remark}

We can now state and prove some useful pointwise estimates for the decay of
$\widehat{\chi}_{C_{\gamma}}\left(  \xi\right)  .$ See \cite{BRT}.

\begin{theorem}
\label{decadCgamma}Let $\gamma>2$ and let $C_{\gamma}$ be as in the
Introduction, let $\psi\in\left(  -\pi/2,\pi/2\right]  $, let either $\theta=\psi
-\pi/2$ or $\theta=\psi+\pi/2$ and let $\Theta=\left(
\cos  \theta ,\sin  \theta \right)  $. Then, for
$\rho\geqslant2$ we have (for small $\varepsilon>0$ and suitable positive
constants $c,c_{1}$)
\begin{equation}
\left\vert \widehat{\chi}_{C_{\gamma}}\left(  \rho\Theta\right)  \right\vert
\leqslant\left\{
\begin{array}
[c]{lll}%
c\rho^{-1-1/\gamma} &  & \text{for }0\leqslant\left\vert \psi\right\vert
\leqslant c_{1}\rho^{-1+1/\gamma},\\
c\rho^{-3/2}\left\vert \psi\right\vert ^{\left(  2-\gamma\right)  /\left(
2\gamma-2\right)  } &  & \text{for }c_{1}\rho^{-1+1/\gamma}\leqslant\left\vert
\psi\right\vert \leqslant\varepsilon,\\
c\rho^{-3/2} &  & \text{for }\varepsilon\leqslant\psi\leqslant\pi.
\end{array}
\right.  \label{inequalities}%
\end{equation}

\end{theorem}

This theorem is the basic result in this paper and we are going to write two
proofs of it.

In the first proof we use elementary arguments to estimate the chords
introduced in Theorem \ref{chord}.

In the second proof we apply the divergence theorem to pass from
$\widehat{\chi}_{C_{\gamma}}$ to $\widehat{\mu}_{\gamma}$, where $\mu_{\gamma
}$ is the measure on $\mathbb{R}^{2}$, supported on $\partial C_{\gamma}$,
where it coincides with the arc-lenght measure. Then we use a partition of
unity to split $\partial C_{\gamma}$ into dyadic pieces.

\begin{proof}
[First proof of Theorem \ref{decadCgamma}]Assume $\psi>0$ and let $x_{0}>0$
satisfy $\gamma x_{0}^{\gamma-1}=\tan\psi$, that is $\left(  x_{0}%
,x_{0}^{\gamma}\right)  $ is the point in $\partial C_{\gamma}$ with outward
unit normal $\Theta$. Let $x_{1}<x_{2}$ be the two solutions of the equation%
\begin{equation}
\left\vert x\right\vert ^{\gamma}=x_{0}^{\gamma}+\left(  \rho\cos
\psi  \right)  ^{-1}+\gamma x_{0}^{\gamma-1}\left(  x-x_{0}\right)
\ ,\label{equationgamma}%
\end{equation}
(of course $x_{1}<x_{0}<x_{2}$, while the assumption $\psi>0$ yields
$\left\vert x_{1}\right\vert <x_{2}$). We observe that $\left\vert
\lambda\left(  \rho^{-1},\psi\right)  \right\vert \leqslant cx_{2}$ and we now
estimate $x_{2}$. 
The inequality $0\leqslant\psi\leqslant c_{1}\rho^{-1+1/\gamma}$ implies that
the equation (\ref{equationgamma}) has no solution when $x>\kappa
\rho^{-1/\gamma}$ with a suitably large constant $\kappa$. Indeed since
$x_{0}^{\gamma-1}\approx\psi$ we have $x_{0}\approx\psi^{1/\left(
\gamma-1\right)  }\leqslant c\rho^{-1/\gamma}$ so that%
\begin{align*}
&  x^{\gamma}-x_{0}^{\gamma}-\left(  \rho\cos \psi \right)
^{-1}-\gamma x_{0}^{\gamma-1}\left(  x-x_{0}\right)  \\
&  >x^{\gamma}-c\rho^{-1}-\left(  \rho\cos\psi  \right)
^{-1}-c\rho^{-1+1/\gamma}x\\
&  >\rho^{-1}\left(  \left(  \rho^{1/\gamma}x\right)  ^{\gamma}-c-\left(
\cos \psi  \right)  ^{-1}-c\rho^{1/\gamma}x\right)  >0
\end{align*}
provided that $\rho^{1/\gamma}x$ is large enough. 
\[
\includegraphics[width=10cm]{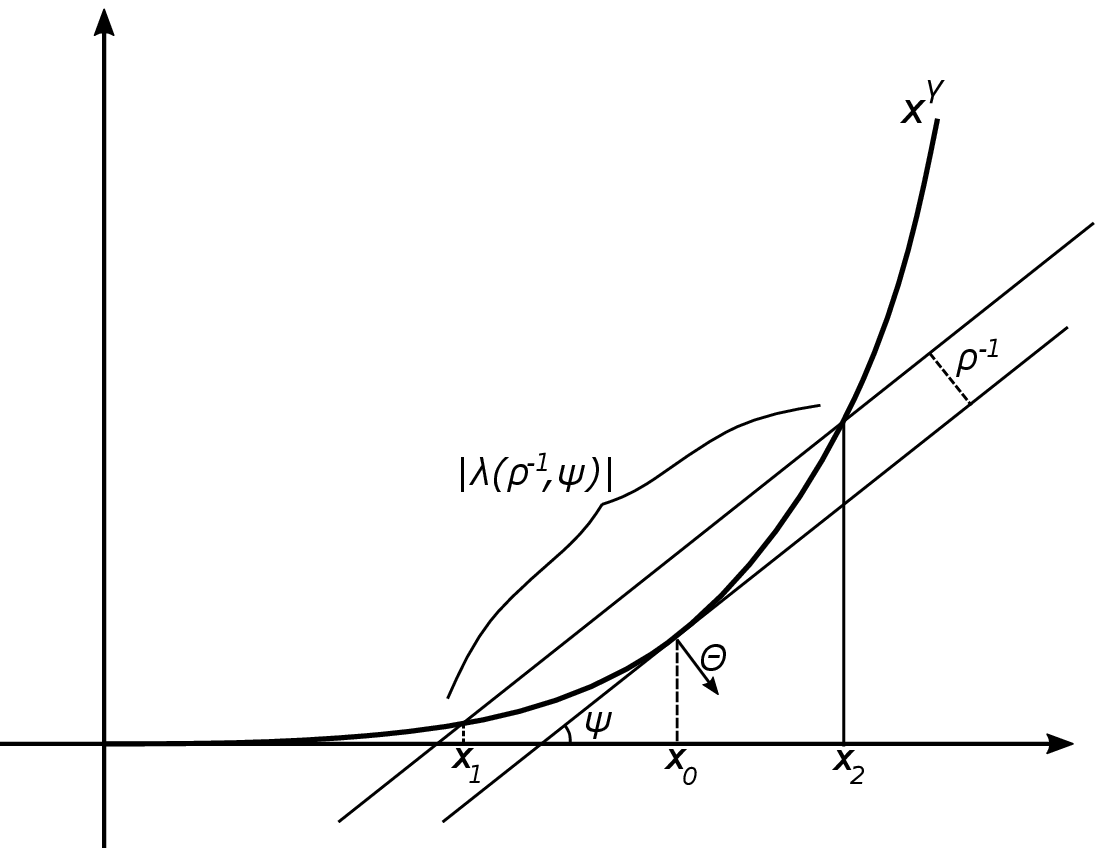}
\]
Let us now assume
$c\rho^{-1+1/\gamma}\leqslant\psi\leqslant\varepsilon$ with a suitable
constant $c$. Since $x_{0}^{\gamma-1}\approx\psi$ we have $x_{1}>0$. Indeed,
let%
\[
y\left(  x\right)  =x_{0}^{\gamma}+\left(  \rho\cos \psi
\right)  ^{-1}+\gamma x_{0}^{\gamma-1}\left(  x-x_{0}\right)  \ .
\]
Let $\psi\geqslant\widetilde{c}\ \rho^{-1+1/\gamma}$ (we shall choose
$\widetilde{c}$ later). Then%
\[
y\left(  0\right)  =\left(  1-\gamma\right)  x_{0}^{\gamma}+\left(  \rho
\cos \psi  \right)  ^{-1}\leqslant\left(  1-\gamma\right)
\ c\ \widetilde{c}\ \rho^{-1}+\left(  \rho\cos \psi  \right)
^{-1}<0
\]
if $\widetilde{c}$ is large enough. Then we observe that, assuming $\left\vert
x-x_{0}\right\vert \geqslant c^{\prime}\rho^{-1/2}x_{0}^{1-\gamma/2}$ with a
suitable choice of $c^{\prime}$, we obtain%
\begin{align*}
&  x^{\gamma}-x_{0}^{\gamma}-\left(  \rho\cos \psi  \right)
^{-1}-\gamma x_{0}^{\gamma-1}\left(  x-x_{0}\right)  \\
&  =\left(  x_{0}+\left(  x-x_{0}\right)  \right)  ^{\gamma}-x_{0}^{\gamma
}-\left(  \rho\cos \psi  \right)  ^{-1}-\gamma x_{0}^{\gamma
-1}\left(  x-x_{0}\right)  \\
&  =x_{0}^{\gamma}\left(  \left(  1+\frac{x-x_{0}}{x_{0}}\right)  ^{\gamma
}-\gamma\frac{x-x_{0}}{x_{0}}-1\right)  -\left(  \rho\cos \psi
\right)  ^{-1}\\
&  \geqslant x_{0}^{\gamma}\ \frac{\gamma}{2}\left(  \frac{x-x_{0}}{x_{0}%
}\right)  ^{2}-\left(  \rho\cos \psi  \right)  ^{-1}%
\geqslant\rho^{-1}\left(  c\ c^{\prime}\ \frac{\gamma}{2}-\left(  \cos\psi  
\right)  ^{-1}\right)  >0\ ,
\end{align*}
since $\frac{x-x_{0}}{x_{0}}>-1$. Observe that we have used the inequality 
\[
\left(
1+u\right)  ^{\gamma}-\gamma u-1\geqslant\gamma u^{2}/2 \;.
\] 
Then $\left\vert
x-x_{0}\right\vert \leqslant c\rho^{-1/2}x_{0}^{1-\gamma/2}$ for every
$x_{1}\leqslant x\leqslant x_{2}$. Therefore%
\[
\left\vert \lambda\left(  \rho,\psi\right)  \right\vert \leqslant c\rho
^{-1/2}x_{0}^{1-\gamma/2}\leqslant c\rho^{-1/2}\psi^{\left(  2-\gamma\right)
/\left(  2\gamma-2\right)  }\ .
\]
Finally let $\varepsilon\leqslant\psi\leqslant\pi$. Then Remark \ref{pococurv}
yields $\left\vert \lambda\left(  \rho,\psi\right)  \right\vert \leqslant
c\rho^{-1/2}$. Collecting the above results and applying Theorem \ref{chord}
we complete the proof.
\end{proof}

For the second proof of Theorem \ref{decadCgamma} we need some well-known lemmas
(see e.g. \cite{Kra},\cite{Mon}, \cite{Ste}). 

\begin{lemma}
\label{Lemma f1}Let $f\in C^{1}\left(  \left[  a,b\right]  \right)  $ be a
convex function such that%
\[
f^{\prime}\left(  x\right)  \geqslant\lambda>0
\]
and let $\varphi$ be a smooth function $\left[  a,b\right]  $. Then%
\[
\left\vert \int_{a}^{b}e^{2\pi if\left(  x\right)  }\varphi\left(  x\right)
dx\right\vert \leqslant\frac{1}{\lambda}\left[  \left\vert \varphi\left(
b\right)  \right\vert +\int_{a}^{b}\left\vert \varphi^{\prime}\left(
x\right)  \right\vert dx\right]  \ .
\]

\end{lemma}

\begin{proof}
Integration by parts yields%
\begin{align*}
\int_{a}^{b}e^{2\pi if\left(  x\right)  }dx &  =\int_{a}^{b}\frac{1}{2\pi
if^{\prime}\left(  x\right)  }\frac{d}{dx}\left(  e^{2\pi if\left(  x\right)
}\right)  dx\\
&  =\frac{1}{2\pi if^{\prime}\left(  b\right)  }e^{2\pi if\left(  b\right)
}-\frac{1}{2\pi if^{\prime}\left(  a\right)  }e^{2\pi if\left(  a\right)  }\\
&  -\int_{a}^{b}\frac{d}{dx}\left(  \frac{1}{2\pi if^{\prime}\left(  x\right)
}\right)  e^{2\pi if\left(  x\right)  }dx\ .
\end{align*}
Hence%
\begin{align*}
\left\vert \int_{a}^{b}e^{2\pi if\left(  x\right)  }dx\right\vert  &
\leqslant\frac{1}{2\pi f^{\prime}\left(  b\right)  }+\frac{1}{2\pi f^{\prime
}\left(  a\right)  }-\frac{1}{2\pi}\int_{a}^{b}\frac{d}{dx}\left(  \frac
{1}{f^{\prime}\left(  x\right)  }\right)  dx\\
&  =\frac{1}{2\pi f^{\prime}\left(  b\right)  }+\frac{1}{2\pi f^{\prime
}\left(  a\right)  }+\frac{1}{2\pi}\frac{1}{f^{\prime}\left(  b\right)
}-\frac{1}{2\pi}\frac{1}{f^{\prime}\left(  a\right)  }\\
&  =\frac{1}{\pi f^{\prime}\left(  b\right)  }\leqslant\frac{1}{\lambda}\ .
\end{align*}
Let now%
\[
G\left(  x\right)  =\int_{a}^{x}e^{2\pi if\left(  t\right)  }dt\ .
\]
Then%
\[
\int_{a}^{b}e^{2\pi if\left(  x\right)  }\varphi\left(  x\right)  dx=\left[
G\left(  x\right)  \varphi\left(  x\right)  \right]  _{a}^{b}-\int_{a}%
^{b}G\left(  x\right)  \varphi^{\prime}\left(  x\right)  dx
\]
and therefore%
\begin{align*}
\left\vert \int_{a}^{b}e^{2\pi if\left(  x\right)  }\varphi\left(  x\right)
dx\right\vert  &  \leqslant\left\vert G\left(  b\right)  \varphi\left(
b\right)  \right\vert +\int_{a}^{b}\left\vert G\left(  x\right)  \right\vert
\left\vert \varphi^{\prime}\left(  x\right)  \right\vert dx\\
&  \leqslant\frac{1}{\lambda}\left\vert \varphi\left(  b\right)  \right\vert
+\frac{1}{\lambda}\int_{a}^{b}\left\vert \varphi^{\prime}\left(  x\right)
\right\vert dx\ .
\end{align*}

\end{proof}

\begin{lemma}
\label{Lemma f2}Let $f\in C^{2}\left(  \left[  a,b\right]  \right)  $ satisfy
$f^{\prime\prime}\left(  x\right)  \geqslant\kappa>0$ and let $\varphi$ be a
smooth function on $\left[  a,b\right]  $. Then%
\[
\left\vert \int_{a}^{b}e^{2\pi if\left(  x\right)  }\varphi\left(  x\right)
dt\right\vert \leqslant\frac{4\left\Vert \varphi\right\Vert _{\infty}}%
{\sqrt{\kappa}}+\frac{2\left\Vert \varphi^{\prime}\right\Vert _{1}}%
{\sqrt{\kappa}}\ .
\]

\end{lemma}

\begin{proof}
Let%
\[
I_{1}=\left\{  x\in\left[  a,b\right]  :\left\vert f^{\prime}\left(  x\right)
\right\vert \leqslant\sqrt{\kappa}\right\}
\]
and%
\[
I_{2}=\left\{  x\in\left[  a,b\right]  :\left\vert f^{\prime}\left(  x\right)
\right\vert >\sqrt{\kappa}\right\}  \ .
\]
The convexity of $f\left(  x\right)  $ implies that $I_{1}$ is either an
interval or the empty set. $I_{2}$ is the union of at most two intervals. Let
$I_{1}=\left[  \alpha,\beta\right]  $. Then the mean value theorem yields%
\[
\left(  \beta-\alpha\right)  \kappa\leqslant f^{\prime}\left(  \beta\right)
-f^{\prime}\left(  \alpha\right)  \leqslant2\sqrt{\kappa}\ .
\]
Hence%
\[
\int_{\alpha}^{\beta}e^{2\pi if\left(  x\right)  }\varphi\left(  x\right)
dt\leqslant\left(  \beta-\alpha\right)  \left\Vert \varphi\right\Vert
_{\infty}\leqslant\frac{2\left\Vert \varphi\right\Vert _{\infty}}{\sqrt
{\kappa}}\ .
\]
To end the proof we observe that the previous lemma yields%
\[
\left\vert \int_{I_{2}}e^{2\pi if\left(  x\right)  }\varphi\left(  x\right)
dt\right\vert \leqslant\frac{2}{\sqrt{\kappa}}\left\Vert \varphi\right\Vert
_{\infty}+\frac{2}{\sqrt{\kappa}}\left\Vert \varphi^{\prime}\right\Vert
_{1}\ .
\]

\end{proof}

\begin{lemma}
\label{Lemma J}Let $\epsilon\in C^{1}\left(  \mathbb{R}\right)  $ such that
$\epsilon\left(  x\right)  \equiv0$ for $\left\vert x\right\vert <\frac{1}{2}$
and $\left\vert x\right\vert \geqslant1$. Then
\[
\left\vert \int_{-\infty}^{+\infty}e^{-2\pi i\left(  au+b\left\vert
u\right\vert ^{\gamma}\right)  }\epsilon\left(  u\right)  \ du\right\vert
\leqslant\frac{c}{\left(  1+\left\vert \left(  a,b\right)  \right\vert
\right)  ^{1/2}}%
\]
(where $c$ is independent of $a,b$, but depends on $\left\Vert \epsilon
\right\Vert _{\infty}$ and $\left\Vert \epsilon^{\prime}\right\Vert _{\infty}$).
\end{lemma}

\begin{proof}
It is enough to consider the integral on $\left(  0,+\infty\right)  $. Let
$f\left(  u\right)  =au+bu^{\gamma}$ and let%
\[
J\left(  a,b\right)  =\int_{0}^{+\infty}e^{-2\pi if\left(  u\right)  }%
\epsilon\left(  u\right)  \ du\ .
\]
If $\left\vert \left(  a,b\right)  \right\vert \leqslant1$ we have the trivial
estimate%
\[
\left\vert J\left(  a,b\right)  \right\vert \leqslant\int_{1/2}^{1}\left\vert
\epsilon\left(  u\right)  \right\vert \ du\leqslant\frac{1}{2}\left\Vert
\epsilon\right\Vert _{\infty}\ .
\]
Assume $\left\vert \left(  a,b\right)  \right\vert >1$ and $\gamma\left\vert
b\right\vert \leqslant\frac{1}{2}\left\vert a\right\vert $. Then%
\[
\left\vert f^{\prime}\left(  u\right)  \right\vert =\left\vert a+b\gamma
u^{\gamma-1}\right\vert \geqslant\left\vert a\right\vert -\gamma\left\vert
b\right\vert \geqslant\frac{1}{2}\left\vert a\right\vert
\]
so that, by Lemma \ref{Lemma f1},%
\[
\left\vert J\left(  a,b\right)  \right\vert \leqslant2\frac{\left\Vert
\epsilon\right\Vert _{\infty}+\left\Vert \epsilon^{\prime}\right\Vert
_{\infty}}{\left\vert a\right\vert }\leqslant\frac{c_{2}}{\left\vert \left(
a,b\right)  \right\vert }\leqslant\frac{c_{2}}{\left\vert \left(  a,b\right)
\right\vert ^{1/2}}\ .
\]
Finally if $\gamma\left\vert b\right\vert >\frac{1}{2}\left\vert a\right\vert
$ then%
\[
\left\vert f^{\prime\prime}\left(  u\right)  \right\vert =\left\vert
b\gamma\left(  \gamma-1\right)  u^{\gamma-2}\right\vert \geqslant
c_{3}\left\vert b\right\vert
\]
so that by Lemma \ref{Lemma f2}%
\[
\left\vert J\left(  a,b\right)  \right\vert \leqslant c_{4}\frac{\left\Vert
\epsilon\right\Vert _{\infty}+\left\Vert \epsilon^{\prime}\right\Vert
_{\infty}}{\left\vert b\right\vert ^{1/2}}\leqslant\frac{c_{5}}{\left\vert
\left(  a,b\right)  \right\vert ^{1/2}}\ .
\]

\end{proof}

\begin{proof}
[Second proof of Theorem \ref{decadCgamma}]For $t\in\mathbb{R}^{2}$ let
$\eta\left(  t\right)  $ be a smooth function supported in a disc $U$ centred
at the origin and such that $\eta\left(  t\right)  =1$ for each $t\in\frac
{1}{2}U$. Observe that for $U$ small enough
\[
\partial C_{\gamma}\cap U=\left\{  \left(  t_{1},t_{2}\right)  \in
\mathbb{R}^{2}:t_{2}=\left\vert t_{1}\right\vert ^{\gamma}\right\}  \cap U\ .
\]
For $t,\xi\in\mathbb{R}^{2}$, let%
\[
\omega\left(  t\right)  =\frac{e^{-2\pi it\cdot\xi}}{-2\pi i\left\vert
\xi\right\vert ^{2}}\xi\ ,
\]
so that%
\[
\mathrm{\mathrm{div}}\omega\left(  t\right)  =\frac{\partial\omega_{1}%
}{\partial t_{1}}+\frac{\partial\omega_{2}}{\partial t_{2}}=e^{-2\pi
it\cdot\xi}\ .
\]
Let us write $\xi=\rho\Theta$ in polar coordinates and for every point
$t\in\partial C_{\gamma}$ let $\nu\left(  t\right)  $ be the outward unit
normal. Then application of the divergence theorem yields%
\begin{align}
\widehat{\chi}_{C_{\gamma}}\left(  \xi\right)   &  =\int_{C_{\gamma}}e^{-2\pi
i\xi\cdot t}\ dt\label{divergenza}\\
&  =\int_{C_{\gamma}}\mathrm{\mathrm{div}}\omega\left(  t\right)
\ dt\nonumber\\
&  =\frac{-1}{2\pi i\rho}\int_{\partial C_{\gamma}}e^{-2\pi i\rho\Theta\cdot
t}\Theta\cdot\nu\left(  t\right)  \ d\mu_{\gamma}\left(  t\right) \nonumber\\
&  =\frac{-1}{2\pi i\rho}\int_{\partial C_{\gamma}}e^{-2\pi i\rho\Theta\cdot
t}\Theta\cdot\nu\left(  t\right)  \eta\left(  t\right)  \ d\mu_{\gamma}\left(
t\right) \nonumber\\
&  -\frac{1}{2\pi i\rho}\int_{\partial C_{\gamma}}e^{-2\pi i\rho\Theta\cdot
t}\Theta\cdot\nu\left(  t\right)  \left(  1-\eta\left(  t\right)  \right)
\ d\mu_{\gamma}\left(  t\right) \nonumber\\
&  =:\frac{-1}{2\pi i\rho}H_{1}\left(  \xi\right)  -\frac{1}{2\pi i\rho}%
H_{2}\left(  \xi\right)  \ .\nonumber
\end{align}
where $\mu_{\gamma}$ is the arc-length measure on $\partial C_{\gamma}$.

We first estimate $H_{2}\left(  \xi\right)  $. Let%
\[
s\mapsto\Gamma\left(  s\right)
\]
be the parametrization of $\partial C_{\gamma}$ by its arc-length. Then
\[
H_{2}\left(  \xi\right)  =\int_{a}^{b}e^{-2\pi i\rho\Theta\cdot\Gamma\left(
s\right)  }\Theta\cdot\nu\left(  \Gamma\left(  s\right)  \right)  \left(
1-\eta\left(  \Gamma\left(  s\right)  \right)  \right)  ds\ .
\]
Since $\Gamma^{\prime}\left(  s\right)  $ and $\Gamma^{\prime\prime}\left(
s\right)  $ are orthogonal vectors with norms $\geqslant c_{1}>0$ then either%
\[
\left\vert \frac{d}{ds}\left(  \rho\Theta\cdot\Gamma\left(  s\right)  \right)
\right\vert \geqslant c_{2}\rho
\]
or%
\[
\left\vert \frac{d^{2}}{ds^{2}}\left(  \rho\Theta\cdot\Gamma\left(  s\right)
\right)  \right\vert \geqslant c_{2}\rho\ .
\]
Therefore we can split the integral in $H_{2}\left(  \xi\right)  $ as the sum of a 
finite number of integrals that satisfy either the assumption of Lemma
\ref{Lemma f1} or Lemma \ref{Lemma f2}. Hence%
\[
\left\vert H_{2}\left(  \xi\right)  \right\vert \leqslant c_{2}\rho^{-1/2}.
\]

Let us consider the integral $H_{1}\left(  \xi\right)  $. By our assumption on
the support of $\eta\left(  t\right)  $ we can write%
\[
H_{1}\left(  \xi\right)  =\int_{\mathbb{R}}e^{-2\pi i\left(  \xi_{1}x+\xi
_{2}\left\vert x\right\vert ^{\gamma}\right)  }\delta\left(  x\right)
\tau\left(  x\right)  \ dx\ ,
\]
where $\tau\left(  x\right)  $ is compactly supported and takes value $1$ in a
neighborhood of $0\in\mathbb{R}$ (say $\tau\left(  x\right)  =1$ when
$\left\vert x\right\vert \leqslant1/2$ and $\tau\left(  x\right)  =0$ when
$\left\vert x\right\vert >1$) and $\delta\left(  x\right)  $ is a $C^{2}$
function (recall that $\gamma>2$).

Assume first $\left\vert \xi_{1}\right\vert >\left\vert \xi_{2}\right\vert $.
Since%
\[
\left\vert \frac{d}{dx}\left(  \xi_{1}x+\xi_{2}\left\vert x\right\vert
^{\gamma}\right)  \right\vert =\left\vert \xi_{1}+\gamma\xi_{2}\left\vert
x\right\vert ^{\gamma-1}\mathrm{sign}\left(  x\right)  \right\vert
\approx\left\vert \xi_{1}\right\vert \approx\rho\ ,
\]
by Lemma \ref{Lemma f1} we have%
\[
\left\vert H_{1}\left(  \xi\right)  \right\vert \leqslant\frac{c}{\rho}\ .
\]

Let now $\left\vert \xi_{1}\right\vert <\left\vert \xi_{2}\right\vert $ and
let $\epsilon\left(  x\right)  =\tau\left(  x\right)  -\tau\left(  2x\right)
$. Observe that $\epsilon\left(  x\right)  $ is positive and supported in the
interval $\left(  -1,-1/4\right)  \cup\left(  1/4,1\right)  $. The key step in
the proof is a dyadic decomposition with the change of variables%
\begin{align*}
&  \int_{\mathbb{R}}e^{-2\pi i\left(  \xi_{1}x+\xi_{2}\left\vert x\right\vert
^{\gamma}\right)  }\delta\left(  x\right)  \tau\left(  x\right)
\ dx=\sum_{j=1}^{+\infty}\int_{\mathbb{R}}e^{-2\pi i\left(  \xi_{1}x+\xi
_{2}\left\vert x\right\vert ^{\gamma}\right)  }\delta\left(  x\right)
\epsilon\left(  2^{j}x\right)  \ dx\\
&  =\sum_{j=1}^{+\infty}2^{-j}\int_{\mathbb{R}}e^{-2\pi i\left(  \left(
\xi_{1}2^{-j}\right)  u+\left(  \xi_{2}2^{-\gamma j}\right)  \left\vert
u\right\vert ^{\gamma}\right)  }\delta\left(  2^{-j}u\right)  \epsilon\left(
u\right)  \ du\ .
\end{align*}

By Lemma \ref{Lemma J} we have%
\[
\left\vert \int_{\mathbb{R}}e^{-2\pi i\left(  \left(  \xi_{1}2^{-j}\right)
u+\left(  \xi_{2}2^{-\gamma j}\right)  u^{\gamma}\right)  }\delta\left(
2^{-j}u\right)  \epsilon\left(  u\right)  \ du\right\vert \leqslant c\left[
1+\left\vert \left(  \xi_{1}2^{-j},\xi_{2}2^{-\gamma j}\right)  \right\vert
\right]  ^{-1/2}\ .
\]
Hence%
\[
\left\vert \int_{\mathbb{R}}e^{-2\pi i\left(  \xi_{1}x+\xi_{2}x^{\gamma
}\right)  }\tau\left(  x\right)  \ dx\right\vert \leqslant c\sum
_{j=1}^{+\infty}2^{-j}\left[  1+\left\vert \left(  \xi_{1}2^{-j},\xi
_{2}2^{-\gamma j}\right)  \right\vert \right]  ^{-1/2}\ .
\]
We recall that we are assuming $\left\vert \xi_{2}\right\vert >\left\vert
\xi_{1}\right\vert $, i.e. we are considering only the directions close to be
perpendicular to the part of $\partial C_{\gamma}$ about the origin. Then%
\begin{align}
&  \sum_{j=1}^{+\infty}2^{-j}\left[  1+\left\vert \left(  \xi_{1}2^{-j}%
,\xi_{2}2^{-\gamma j}\right)  \right\vert \right]  ^{-1/2}\label{diadic}\\
&  \leqslant\sum_{2^{j}\leqslant\left(  \left\vert \xi_{2}\right\vert
/\left\vert \xi_{1}\right\vert \right)  ^{1/\left(  \gamma-1\right)  }}%
2^{-j}\left\vert \left(  \xi_{1}2^{-j},\xi_{2}2^{-\gamma j}\right)
\right\vert ^{-1/2}\nonumber\\
&  +\sum_{2^{j}>\left(  \left\vert \xi_{2}\right\vert /\left\vert \xi
_{1}\right\vert \right)  ^{1/\left(  \gamma-1\right)  }}2^{-j}\left\vert
\left(  \xi_{1}2^{-j},\xi_{2}2^{-\gamma j}\right)  \right\vert ^{-1/2}%
\nonumber\\
&  \leqslant c\sum_{2^{j}\leqslant\left(  \left\vert \xi_{2}\right\vert
/\left\vert \xi_{1}\right\vert \right)  ^{1/\left(  \gamma-1\right)  }%
}2^{j\left(  \gamma/2-1\right)  }\left\vert \xi_{2}\right\vert ^{-1/2}%
+c\sum_{2^{j}>\left(  \left\vert \xi_{2}\right\vert /\left\vert \xi
_{1}\right\vert \right)  ^{1/\left(  \gamma-1\right)  }}2^{-j/2}\left\vert
\xi_{1}\right\vert ^{-1/2}\nonumber\\
&  \leqslant c\left\vert \xi_{2}\right\vert ^{-1/2}\left(  \frac{\left\vert
\xi_{2}\right\vert }{\left\vert \xi_{1}\right\vert }\right)  ^{\left(
\gamma-2\right)  /\left(  2\gamma-2\right)  }+\left\vert \xi_{1}\right\vert
^{-1/2}\left(  \frac{\left\vert \xi_{1}\right\vert }{\left\vert \xi
_{2}\right\vert }\right)  ^{1/\left(  2\gamma-2\right)  }\nonumber\\
&  \leqslant c\left\vert \xi_{2}\right\vert ^{-1/2}\left(  \frac{\left\vert
\xi_{2}\right\vert }{\left\vert \xi_{1}\right\vert }\right)  ^{\left(
\gamma-2\right)  /\left(  2\gamma-2\right)  }\nonumber   \\ &\approx\rho^{-1/2}\psi^{-\left(
\gamma-2\right)  /\left(  2\gamma-2\right)  }\ ,\nonumber
\end{align}
where $\psi=\pi/2+\arctan\left(  \left\vert \xi_{2}\right\vert /\left\vert
\xi_{1}\right\vert \right)  $. Hence
\begin{equation}
\left\vert \widehat{\chi}_{C_{\gamma}}\left(  \xi\right)  \right\vert
\leqslant c\rho^{-3/2}\psi^{\left(  2-\gamma\right)  /\left(  2\gamma
-2\right)  }\ .\label{Stima3}%
\end{equation}
Finally we prove the inequality%
\[
\left\vert \widehat{\chi}_{C_{\gamma}}\left(  \rho\Theta\right)  \right\vert
\leqslant c\rho^{-1-1/\gamma}\ .
\]
Observe that (\ref{Stima3}) yields the above upper bound when $\psi\geqslant
c\rho^{-1+1/\gamma}$. We still have to prove that the same bound is correct
when $0\leqslant\psi\leqslant c\rho^{-1+1/\gamma}$, that is $\left\vert
\xi_{1}\right\vert /\left\vert \xi_{2}\right\vert \leqslant c\rho
^{-1+1/\gamma}$. Finally we deal with the first inequality. We can assume
$\left\vert \xi_{1}\right\vert <c\left\vert \xi_{2}\right\vert $. By the
previous computation we have to bound%
\begin{align*}
&  \sum_{j=1}^{+\infty}2^{-j}\left[  1+\left\vert \left(  \xi_{1}2^{-j}%
,\xi_{2}2^{-\gamma j}\right)  \right\vert \right]  ^{-1/2}\\
&  \leqslant c\sum_{j=1}^{+\infty}2^{-j}\left(  1+\left\vert \xi
_{2}\right\vert 2^{-\gamma j}\right)  ^{-1/2}\\
&  \leqslant c\sum_{2^{j}\leqslant\left\vert \xi_{2}\right\vert ^{1/\gamma}%
}2^{-j}\left(  \left\vert \xi_{2}\right\vert 2^{-\gamma j}\right)
^{-1/2}+c\sum_{2^{j}>\left\vert \xi_{2}\right\vert ^{1/\gamma}}2^{-j}\leqslant
c\left\vert \xi_{2}\right\vert ^{-1/\gamma}\ ,
\end{align*}
which yields the first inequality in (\ref{inequalities}).
\end{proof}
\section{Average decay of $\widehat{\chi}_{C_{\gamma}}\left(  \xi\right)  $}

We shall consider both $L^{p}$ average discrepancies when $C_{\gamma}$ is
translated, and $L^{p}$ average discrepancies when $C_{\gamma}$ is translated
and rotated. For the latter problem we shall need estimates for the $L^{p}$
\textit{(spherical) average decay }of $\widehat{\chi}_{C_{\gamma}}\left(
\xi\right)  $, that is%
\[
\left\{  \int_{0}^{2\pi}\left\vert \widehat{\chi}_{C_{\gamma}}\left(
\rho\Theta\right)  \right\vert ^{p}\ d\theta\right\}  ^{1/p}%
\]
(where $\Theta=\left(  \cos  \theta  ,\sin  \theta\right)  
$ and $\rho\geqslant 2$). To illustrate the relevance of these averages we
point out that the above estimate (\ref{L uno discrepanza}) for the
discrepancy of a polygon $P$ is a consequence of the estimate%
\[
\int_{0}^{2\pi}\left\vert \widehat{\chi}_{P}\left(  \rho\Theta\right)
\right\vert \ d\theta\leqslant c\frac{\log^{2}\left(  \rho\right)  }{\rho^{2}%
}\ ,
\]
which in turn follows from Theorem \ref{chord}. We refer the interested reader
to \cite{BCT}, \cite{BGT} and \cite{Tra0} for more details and applications.

In the next theorem (see \cite{BRT}) we obtain estimates for the $L^{p}$
\textit{(spherical) average decay }of $\widehat{\chi}_{C_{\gamma}}\left(
\xi\right) $.
\begin{theorem}
\label{decadmedio}We have%
\begin{align*}
&  \left\{  \int_{0}^{2\pi}\left\vert \widehat{\chi}_{C_{\gamma}}\left(
  \rho\Theta\right)   \right\vert ^{p}\ d\theta\right\}
^{1/p}\\
&  \leqslant\left\{
\begin{array}
[c]{lll}%
c_{p}\rho^{-3/2} &  & \text{for }p<\frac{2\gamma-2}{\gamma-2},\\
c\rho^{-3/2}\log^{\left(  \gamma-2\right)  \left(  2\gamma-2\right)  }\left(
\rho\right)   &  & \text{for }p=\frac{2\gamma-2}{\gamma-2},\\
c_{p}\rho^{-1-\frac{1}{p}-\frac{1}{\gamma}+\frac{1}{\gamma p}} &  & \text{for
}p>\frac{2\gamma-2}{\gamma-2}.
\end{array}
\right.
\end{align*}

\end{theorem}

\begin{proof}
It is enough to integrate between $-\pi/2$ and $\pi/2$. The estimates in
Theorem \ref{decadCgamma} yield%
\begin{align*}
&  \left\{  \int_{-\pi/2}^{\pi/2}\left\vert \widehat{\chi}_{C_{\gamma}}\left(
  \rho\Theta\right)    \right\vert ^{p}\ d\theta\right\}
^{1/p}\\
&  \leqslant\left\{  \int_{-\pi/2}^{-\pi/2+c\rho^{-1+1/\gamma}}\left\vert
\widehat{\chi}_{C_{\gamma}}\left(    \rho\Theta \right)
\right\vert ^{p}\ d\theta\right\}  ^{1/p}\\
&  +\left\{  \int_{-\pi/2+c\rho^{-1+1/\gamma}}^{-\pi/2+\varepsilon}\left\vert
\widehat{\chi}_{C_{\gamma}}\left(   \rho\Theta\right)  
\right\vert ^{p}\ d\theta\right\}  ^{1/p}\\
&  +\left\{  \int_{-\pi/2+\varepsilon}^{\pi/2}\left\vert \widehat{\chi
}_{C_{\gamma}}\left(  \rho\Theta\right)   \right\vert
^{p}\ d\theta\right\}  ^{1/p}\\
&  \leqslant c\left\{  \int_{0}^{c\rho^{-1+1/\gamma}}\left\vert \rho
^{-1-1/\gamma}\right\vert ^{p}\ d\theta\right\}  ^{1/p}\\
&  +c\left\{  \int_{c\rho^{-1+1/\gamma}}^{\varepsilon}\left\vert \rho
^{-3/2}\psi^{\left(  2-\gamma\right)  /\left(  2\gamma-2\right)  }\right\vert
^{p}\ d\psi\right\}  ^{1/p}\\
&  +c\left\{  \int_{-\pi/2+\varepsilon}^{\pi/2}\left\vert \rho^{-3/2}%
\right\vert ^{p}\ d\theta\right\}  ^{1/p}\\
&  =A+B+C\ ,
\end{align*}
say. Finally we have%
\begin{align*}
A &  \leqslant c\rho^{-1-\frac{1}{p}-\frac{1}{\gamma}+\frac{1}{\gamma p}}\ ,\\
B &  \leqslant\left\{
\begin{array}
[c]{lll}%
\rho^{-3/2} &  & \textit{for }p<\frac{2\gamma-2}{\gamma-2},\\
\rho^{-3/2}\log^{\left(  \gamma-2\right)  /\left(  2\gamma-2\right)  }\left(
\rho\right)   &  & \textit{for }p=\frac{2\gamma-2}{\gamma-2},\\
\rho^{-1-\frac{1}{p}-\frac{1}{\gamma}+\frac{1}{\gamma p}} &  & \textit{for
}p>\frac{2\gamma-2}{\gamma-2},
\end{array}
\right.  \\
C &  \leqslant c\rho^{-3/2}\ .
\end{align*}

\end{proof}

It can be proved that the above estimates are sharp (see \cite{BRT}).

\section{Integer points in $C_{\gamma}$}

We consider two different averages of the discrepancy function.

\subsection{Discrepancy over translations}

We now prove a few $L^p$ estimates for the discrepancy function%
\begin{align*}
\mathcal{D}_{R}\left(  t\right)   &  =\mathcal{D}\left(  RC_{\gamma}+t\right)
=-R^{d}\left\vert C_{\gamma}\right\vert +\mathrm{card}\left(  (RC_{\gamma
}+t)\cap\mathbb{Z}^{2}\right)  \\
&  =-R^{2}\left\vert C_{\gamma}\right\vert +\sum_{n\in\mathbb{Z}^{2}}%
\chi_{RC_{\gamma}}\left(  n-t\right)  \ ,
\end{align*}
which we recall to have Fourier series%
\[
\sum_{0\neq m\in\mathbb{Z}^{2}}\widehat{\chi}_{RC_{\gamma}}\left(  m\right)
e^{2\pi im\cdot t}\ .
\]
We consider the $L^{p}$ norms%
\[
\left\Vert \mathcal{D}_{R}\right\Vert _{p}=\left\{
\begin{array}
[c]{lll}%
\left\{  \int_{\mathbb{T}^{2}}\left\vert \mathcal{D}_{R}\left(  RC_{\gamma
}+t\right)  \right\vert ^{p}dt\right\}  ^{1/p} &  & \textit{for }p<\infty,\\
\\
\sup_{t\in\mathbb{T}^{2}}\left\vert \mathcal{D}_{R}\left(  RC_{\gamma
}+t\right)  \right\vert  &  & \textit{for }p=\infty.
\end{array}
\right.
\]
Our estimates are the following (see \cite{BCGGT}).

\begin{theorem}
\label{minore dpiu1}For $2<\gamma\leqslant3$ we have%
\begin{equation}
\left\Vert \mathcal{D}_{R}\right\Vert _{p}\leqslant\left\{
\begin{array}
[c]{lll}%
cR^{1-1/\gamma} &  & \text{for }1\leqslant p\leqslant4/\left(  3-\gamma
\right)  ,\\
cR^{\frac{2}{3}\left(  1-\frac2{\gamma p}\right)  } &  & \text{for
}p>4/\left(  3-\gamma\right)  .
\end{array}
\right.  \label{colin}%
\end{equation}

\end{theorem}

\begin{theorem}
\label{maggiore dpiu1}For $\gamma>3$ and every $p\geqslant1$ we have%
\[
\left\Vert \mathcal{D}_{R}\right\Vert _{p}\ \leqslant cR^{1-1/\gamma}\ .
\]

\end{theorem}

\begin{remark}
The proof of Theorem \ref{minore dpiu1} follows Hlawka's  smoothing
argument that is usually used when the curvature
of the boundary is strictly positive (that is $\gamma=2$). Anyway it takes no
extra effort to apply it to the case $\gamma\leqslant3$. Roughly speaking here
we have to consider two cases. First, the integer points close to the origin,
where vertical translations yield discrepancy $\leqslant cR^{1-1/\gamma}$.
Second, the integer points away from the origin, where the smoothing argument
yields discrepancy $\leqslant cR^{2/3}$. Therefore $\gamma\leqslant3$ works as
well. The bound $cR^{2/3}$ for $\gamma\leqslant3$ has been first obtained in
\cite{Col}.
\end{remark}

We need the following lemma (see \cite{BGT}).

\begin{lemma}
\label{ColzLemma} Let $\varphi\left(  t\right)  $ be a smooth non-negative
function supported in a small neighbourhood of the origin and such that
$\int_{\mathbb{R}^{2}}\varphi=1$. Then for every small $\varepsilon>0$ and
$R>1$ we have%
\[
\varepsilon^{-2}\varphi\left(  \varepsilon^{-1}\cdot\right)  \ast\chi_{\left(
R-\varepsilon\right)  C_{\gamma}}\left(  t\right)  \leqslant
\chi_{RC_{\gamma}}\left(  t\right)  \leqslant\varepsilon^{-2}%
\varphi\left(  \varepsilon^{-1}\cdot\right)  \ast\chi_{\left(  R+\varepsilon
\right)  C_{\gamma}}\left(  t\right)  \ ,
\]
where $\ast$ denotes the convolution%
\[
\left(  f\ast g\right)  \left(  t\right)  =\int f\left(  t-s\right)  g\left(
s\right)  \ ds \;.
\]
In particular,%
\begin{align}
&  \left\vert C_{\gamma}\right\vert \left(  \left(  R-\varepsilon\right)
^{2}-R^{2}\right)  +D_{\varepsilon,R-\varepsilon}\left(  t\right)
\label{D hlawka}\\
&  \leqslant\mathcal{D}_{R}\left(  t\right)  \leqslant\left\vert C_{\gamma
}\right\vert \left(  \left(  R+\varepsilon\right)  ^{2}-R^{2}\right)
+D_{\varepsilon,R+\varepsilon}\left(  t\right)  \ ,\nonumber
\end{align}
where%
\[
D_{\varepsilon,R}\left(  t\right)  =R^{2}\sum_{0\neq m\mathbf{\in}%
\mathbb{Z}^{2}}\widehat{\varphi}\left(  \varepsilon m\right)  \widehat{\chi
}_{C_{\gamma}}\left(  Rm\right)  e^{2\pi im\cdot t}\ .
\]

\end{lemma}

\begin{proof}
First we observe that the convexity of $C_{\gamma}$ yields%
\[
\frac{R}{R+\varepsilon}C_{\gamma}+\frac{\varepsilon}{R+\varepsilon}C_{\gamma
}\subseteq C_{\gamma}%
\]
so that%
\begin{equation}
\left(  R+\varepsilon\right)  C_{\gamma}\supseteq RC_{\gamma}+\varepsilon
C_{\gamma}\supseteq RC_{\gamma}+B\left(  0,\varepsilon\right)  \label{above}%
\end{equation}
and therefore $B\left(  q,\varepsilon\right)  \subseteq\left(  R+\varepsilon
\right)  C_{\gamma}$ for every $q\in\partial\left(  RC_{\gamma}\right)  $.
Applying (\ref{above}) to $\mathrm{Interior}\left(  C_{\gamma}\right)  $ with
$R$ in place of $R+\varepsilon$ we obtain%
\[
\mathrm{Interior}\left(  RC_{\gamma}\right)  \supseteq\mathrm{Interior}\left(
R-\varepsilon\right)  C_{\gamma}+B\left(  0,\varepsilon\right)  \ .
\]
Assume there exists $y\in B\left(  q,\varepsilon\right)  \cap\mathrm{Interior}%
\left(  R-\varepsilon\right)  C_{\gamma}$. It follows that
\[
q\in\mathrm{Interior}\left(  R-\varepsilon\right)  C_{\gamma}+B\left(
0,\varepsilon\right)  \subseteq\mathrm{Interior}\left(  RC_{\gamma}\right)
\]
so that $q\notin\partial\left(  RC_{\gamma}\right)  $. Hence for large $R$ and
small $\varepsilon$ we have%
\[
B\left(  q,\varepsilon\right)  \subseteq\left(  R+\varepsilon\right)
C_{\gamma}\setminus\mathrm{Interior}\left(  R-\varepsilon\right)  C_{\gamma}%
\]
for every $q\in\partial\left(  RC_{\gamma}\right)  $. Then%
\[
\varepsilon^{-2}\varphi\left(  \varepsilon^{-1}\cdot\right)  \ast\chi_{\left(
R-\varepsilon\right)  B}\left(  t\right)  \leqslant\chi_{RB}\left(  t\right)
\leqslant\varepsilon^{-2}\varphi\left(  \varepsilon^{-1}\cdot\right)  \ast
\chi_{\left(  R+\varepsilon\right)  B}\left(  t\right)
\]
and therefore (\ref{D hlawka}).
\end{proof}

\begin{proof}
[Proof of Theorem \ref{minore dpiu1}]By Lemma \ref{ColzLemma} we have%
\begin{align*}
\left\Vert \mathcal{D}_{R}\right\Vert _{p} &  \leqslant\left\vert C_{\gamma
}\right\vert \max_{\pm}\left\vert \left(  R\pm\varepsilon\right)  ^{2}%
-R^{2}\right\vert +\max_{\pm}\left\Vert D_{\varepsilon,R\pm\varepsilon
}\right\Vert _{p}\\
&  \leqslant cR\varepsilon+\max_{\pm}\left\Vert D_{\varepsilon,R\pm
\varepsilon}\left(  t\right)  \right\Vert _{p}\ .
\end{align*}
We write $m=\left(  m_{1},m_{2}\right)  $ and we choose $\varphi\left(
t\right)  $ as in Lemma \ref{ColzLemma}, so that, in particular,%
\[
\left\vert \widehat{\varphi}\left(  \xi\right)  \right\vert \leqslant
\frac{c_{N}}{1+\left\vert \xi\right\vert ^{N}}%
\]
for every $N$. Then%
\begin{align*}
&  \left\Vert D_{\varepsilon,R}\right\Vert _{p}=\left\{  \int_{\mathbb{T}^{2}%
}\left\vert R^{2}\sum\limits_{0\neq m\mathbf{\in}\mathbb{Z}^{2}}%
\widehat{\varphi}\left(  \varepsilon m\right)  \widehat{\chi}_{C_{\gamma}%
}\left(  Rm\right)  e^{2\pi im\cdot t}\right\vert ^{p}\ dt\right\}  ^{1/p}\\
&  \leqslant\left\{  \int_{\mathbb{T}^{2}}\left\vert R^{2}\sum
\limits_{\left\vert \arctan\left(  m_{1}/m_{2}\right)  \right\vert \leqslant
c_{1}\left\vert m\right\vert ^{-1+1/\gamma}}\widehat{\varphi}\left(
\varepsilon m\right)  \widehat{\chi}_{C_{\gamma}}\left(  Rm\right)  e^{2\pi
im\cdot t}\right\vert ^{p}\ dt\right\}  ^{1/p}\\
&  +\left\{  \int_{\mathbb{T}^{2}}\left\vert R^{2}\sum\limits_{c_{1}\left\vert
m\right\vert ^{-1+1/\gamma}\leqslant\left\vert \arctan\left(  m_{1}%
/m_{2}\right)  \right\vert <c_{2}}\widehat{\varphi}\left(  \varepsilon
m\right)  \widehat{\chi}_{C_{\gamma}}\left(  Rm\right)  e^{2\pi im\cdot
t}\right\vert ^{p}\ dt\right\}  ^{1/p}\\
&  +\left\{  \int_{\mathbb{T}^{2}}\left\vert R^{2}\sum\limits_{c_{2}%
\leqslant\left\vert \arctan\left(  m_{1}/m_{2}\right)  \right\vert }%
\widehat{\varphi}\left(  \varepsilon m\right)  \widehat{\chi}_{C_{\gamma}%
}\left(  Rm\right)  e^{2\pi im\cdot t}\right\vert ^{p}\ dt\right\}  ^{1/p}\\
&  =I+II+III\ ,
\end{align*}
say.  In order to prove the first inequality in (\ref{colin}) it is enough to
consider the case $p=4/\left(  3-\gamma\right)  $ (observe that for $\gamma=3$
we have $p=\infty$). We are going to deduce the estimates of $I,\ II,\ III$
from Theorem \ref{decadCgamma}. We have%
\[
I\leqslant cR^{2}\sum\limits_{\left\vert \arctan\left(  m_{1}/m_{2}\right)
\right\vert \leqslant c_{1}\left\vert m\right\vert ^{-1+1/\gamma}}\frac
{1}{1+\left\vert \varepsilon m\right\vert }\left\vert Rm\right\vert
^{-1-1/\gamma}\ .
\]
A modification of the above constant $c_{1}$ allows us to replace the sum
\[
\sum_{\left\vert \arctan\left(  m_{1}/m_{2}\right)  \right\vert \leqslant
c_{1}\left\vert m\right\vert ^{-1+1/\gamma}}
\]
with an integral, but for a
finite number of unit squares close to the origin and centred on the vertical
axis. We write%
\begin{align*}
I &  \leqslant cR^{1-1/\gamma}+cR^{1-1/\gamma}\int_{1}^{+\infty} \  \ \int
_{0}^{c\left(  R\rho\right)  ^{-1+1/\gamma}}\ d\psi\frac{1}{1+\varepsilon\rho
}\rho^{-1-1/\gamma}\rho\ d\rho\\
&  \leqslant cR^{1-1/\gamma}+c\int_{1}^{+\infty}\rho^{-1}\frac{1}%
{1+\varepsilon\rho}\ d\rho=cR^{1-1/\gamma}+c\log\left(  1/\varepsilon\right)
\ .
\end{align*}
By the Hausdorff-Young inequality we have, for $\frac{1}{p}+\frac{1}{q}=1$,%
\begin{align*}
&II \\
&  \leqslant\left\{  \int_{\mathbb{T}^{2}}\left\vert R^{2}\sum
\limits_{c_{1}\left\vert m\right\vert ^{-1+1/\gamma}\leqslant\left\vert
\arctan\left(  m_{1}/m_{2}\right)  \right\vert <c_{2}}\widehat{\varphi}\left(
\varepsilon m\right)  \widehat{\chi}_{C_{\gamma}}\left(  Rm\right)  e^{2\pi
im\cdot t}\right\vert ^{p}\ dt\right\}  ^{1/p}\\
&  \leqslant cR^{2}\left\{  \sum\limits_{c_{1}\left\vert m\right\vert
^{-1+1/\gamma}\leqslant\left\vert \arctan\left(  m_{1}/m_{2}\right)
\right\vert <c_{2}}\left\vert \widehat{\varphi}\left(  \varepsilon m\right)
\widehat{\chi}_{C_{\gamma}}\left(  Rm\right)  \right\vert ^{q}\right\}
^{1/q}\\
&  \leqslant cR^{2}\left\{  \sum\limits_{c_{1}\left\vert m\right\vert
^{-1+1/\gamma}\leqslant\left\vert \arctan\left(  m_{1}/m_{2}\right)
\right\vert <c_{2}}\right.  \\
&  \left.  \ \ \ \ \ \left\vert \frac{1}{1+\left\vert \varepsilon m\right\vert
}\left\vert Rm\right\vert ^{-3/2}\left\vert \frac{m_{1}}{m_{2}}\right\vert
^{\left(  2-\gamma\right)  /\left(  2\gamma-2\right)  }\right\vert
^{q}\right\}  ^{1/q}\\
&  \leqslant cR^{1/2}\left\{  \int_{1}^{+\infty}\int_{c\rho^{-1+1/\gamma}%
}^{c_{2}}\frac{1}{\left(  1+\varepsilon\rho\right)  ^{q}}\rho^{-3q/2}%
\psi^{q\left(  2-\gamma\right)  /\left(  2\gamma-2\right)  }\ d\psi\rho
d\rho\right\}  ^{1/q}\\
&  =cR^{1/2}\left\{  \int_{1}^{+\infty}\frac{1}{\left(  1+\varepsilon
\rho\right)  ^{q}}\rho^{1-3q/2}\int_{c\rho^{-1+1/\gamma}}^{c_{2}}%
\psi^{q\left(  2-\gamma\right)  /\left(  2\gamma-2\right)  }\ d\psi
d\rho\right\}  ^{1/q}\\
&  =cR^{1/2}\left\{  \int_{1}^{+\infty}\frac{1}{\left(  1+\varepsilon
\rho\right)  ^{q}}\rho^{1-3q/2}\ d\rho\right\}  ^{1/q}\\
&  =cR^{1/2}\left\{  \int_{\varepsilon}^{+\infty}\frac{1}{\left(  1+s\right)
^{q}}\left(  \frac{s}{\varepsilon}\right)  ^{1-3q/2}\frac{1}{\varepsilon
}\ ds\right\}  ^{1/q}\\
&  =cR^{1/2}\varepsilon^{3/2-2/q}\\
&  =cR^{1/2}\varepsilon^{\left(  2-\gamma\right)  /2}\ ,
\end{align*}
because $q=4/\left(  \gamma+1\right)  <4/3$.%
\begin{align*}
III &  \leqslant cR^{2}\left\{  \sum\nolimits_{c_{2}\leqslant\left\vert
\arctan\left(  m_{1}/m_{2}\right)  \right\vert }\left(  \frac{1}%
{1+\varepsilon\left\vert m\right\vert }\left\vert Rm\right\vert ^{-3/2}%
\right)  ^{q}\right\}  ^{1/q}\\
&  \leqslant cR^{1/2}\int_{1}^{+\infty}\frac{1}{\left(  1+\varepsilon
\rho\right)  ^{q}}\rho^{1-3q/2}\ d\rho=cR^{1/2}\varepsilon^{3/2-2/q}%
=cR^{1/2}\varepsilon^{\left(  2-\gamma\right)  /2}\ .
\end{align*}
Then%
\[
\left\Vert \mathcal{D}_{R}\right\Vert _{p}\leqslant cR\varepsilon
+cR^{1-1/\gamma}+c\log\left(  1/\varepsilon\right)  +cR^{1/2}\varepsilon
^{\left(  2-\gamma\right)  /2}\ .
\]
By choosing $\varepsilon=R^{-1/\gamma}$ we obtain%
\[
\left\Vert \mathcal{D}_{R}\right\Vert _{p}\leqslant cR^{1-1/\gamma}\ .
\]
A similar computation shows that%
\[
\left\Vert \mathcal{D}_{R}\right\Vert _{\infty}=\sup_{t}\left\vert
\mathcal{D}\left(  RC_{\gamma}+t\right)  \right\vert \leqslant cR^{2/3}\ .
\]
To end the proof we need to show that$\ \ \left\Vert \mathcal{D}%
_{R}\right\Vert _{p}\leqslant cR^{\left(  2\gamma p-4\right)  /\left(  3\gamma
p\right)  }$ for $4/\left(  3-\gamma\right)  <p<\infty$. Interpolation between
the previous two cases yields%
\begin{align*}
\left\Vert \mathcal{D}_{R}\right\Vert _{p} &  =\left\{  \int_{\mathbb{T}^{2}%
}\left\vert \mathcal{D}_{R}\left(  t\right)  \right\vert ^{p}\ dt\right\}
^{1/p}\\
&  \leqslant\left\{  \int_{\mathbb{T}^{2}}\left\Vert \mathcal{D}%
_{R}\right\Vert _{\infty}^{p-4/\left(  3-\gamma\right)  }\left\vert
\mathcal{D}_{R}\left(  t\right)  \right\vert ^{4/\left(  3-\gamma\right)
}\ dt\right\}  ^{1/p}\\
&  =\left\Vert \mathcal{D}_{R}\right\Vert _{\infty}^{1-4/\left(  3p-\gamma
p\right)  }\left\Vert \mathcal{D}_{R}\right\Vert _{4/\left(  3-\gamma\right)
}^{4/\left(  3p-\gamma p\right)  }\leqslant cR^{\left(  2\gamma p-4\right)
/\left(  3\gamma p\right)  }\ .
\end{align*}

\end{proof}

\begin{proof}
[Proof of Theorem \ref{maggiore dpiu1}]It is enough to consider the case
$p=+\infty$. Arguing as in the previous proof we write $\left\Vert
\mathcal{D}_{R}\right\Vert _{\infty}\leqslant I+II+III$ and we obtain
\[
I\leqslant cR^{1-1/\gamma}\ ,\ \ \ \ \ \ \ \ \ \ II\leqslant R^{1/2}%
\varepsilon^{-1/2}\ ,\ \ \ \ \ \ \ \ \ \ III\leqslant R^{1/2}\varepsilon
^{-1/2}\ .
\]
Since now $1-1/\gamma>2/3$, choosing $\varepsilon=R^{-1+2/\gamma}$\ we obtain%
\[
\left\Vert \mathcal{D}_{R}\right\Vert _{\infty}\leqslant cR^{1-1/\gamma}\ .
\]

\end{proof}

\subsection{Discrepancy over translations and rotations}

We obtain better estimates by averaging the discrepancy over translations and
rotations. Here is a result from \cite{Gar}.

\begin{theorem}
\label{Gari} Let $2<\gamma\leqslant3$ and $p<4$ (hence $p\leqslant\left(
2\gamma-2\right)  /\left(  \gamma-2\right)  $). Then%
\begin{equation}
\left\{  \int_{SO\left(  2\right)  }\int_{\mathbb{T}^{2}}\left\vert
\mathcal{D}\left(  R\sigma\left(  C_{\gamma}\right)  +t\right)  \right\vert
^{p}dtd\sigma\right\}  ^{1/p}\leqslant c\ R^{1/2}\ ,\label{L2bianca}%
\end{equation}
where the constant $c$ depends on $\gamma$ and on $p$.
\end{theorem}

\begin{proof}
Let $q$ be the conjugate index of $p$ (that is $1/p+1/q=1$). By the
inequalities of Hausdorff-Young and Minkowski, and by Theorem \ref{decadmedio}
we have%
\begin{align*}
&  \left\{  \int_{SO\left(  2\right)  }\int_{\mathbb{T}^{2}}\left\vert
\mathcal{D}\left(  R\sigma\left(  C_{\gamma}\right)  +t\right)  \right\vert
^{p}dtd\sigma\right\}  ^{1/p}\\
&  =\left\{  \int_{SO\left(  2\right)  }\left[  \left(  R^{2}\int
_{\mathbb{T}^{2}}\left\vert \sum_{m\neq0}\widehat{\chi}_{C_{\gamma}}\left(
R\sigma\left(  m\right)  \right)  e^{2\pi im\cdot t}\right\vert ^{p}dt\right)
^{1/p}\right]  ^{p}d\sigma\right\}  ^{1/p}\\
&  \leqslant R^{2}\left\{  \int_{SO\left(  2\right)  }\left\{  \sum_{m\neq
0}\left\vert \widehat{\chi}_{C_{\gamma}}\left(  R\sigma\left(  m\right)
\right)  \right\vert ^{q}\right\}  ^{p/q}d\sigma\right\}  ^{1/p}\\
&  =R^{2}\left(  \left\Vert \sum_{m\neq0}\left\vert \widehat{\chi}_{C_{\gamma
}}\left(  R\sigma\left(  m\right)  \right)  \right\vert ^{q}\right\Vert
_{L^{p/q}\left(  SO\left(  2\right)  \right)  }\right)  ^{1/q}\\
&  \leqslant R^{2}\left(  \sum_{m\neq0}\left\Vert \left\vert \widehat{\chi
}_{C_{\gamma}}\left(  R\sigma\left(  m\right)  \right)  \right\vert
^{q}\right\Vert _{L^{p/q}\left(  SO\left(  2\right)  \right)  }\right)
^{1/q}\\
&  \leqslant R^{2}\left\{  \sum_{m\neq0}\left\{  \int_{SO\left(  2\right)
}\left\vert \widehat{\chi}_{C_{\gamma}}\left(  R\sigma\left(  m\right)
\right)  \right\vert ^{p}d\sigma\right\}  ^{q/p}\right\}  ^{1/q}\\
&  \leqslant cR^{2}\left\{  \sum_{m\neq0}\left\vert Rm\right\vert
^{-3q/2}\right\}  ^{1/q}=cR^{1/2}\ ,
\end{align*}
because $q>4/3$.
\end{proof}

It is known that (\ref{L2bianca}) can be reversed (see \cite{BRT} for a proof). Here we
propose a different proof which depends on a general argument. We need a few
preliminary results which are essentially known (see \cite{Ste} and \cite{Hor}).

\begin{proposition}
\label{Asymptotic}Let $\phi\in C^{\infty}\left(  -\infty,+\infty\right)  $ be
a convex function such that $\phi\left(  0\right)  =\phi^{\prime}\left(
0\right)  =0$, $\phi^{\prime\prime}\left(  0\right)  >0$. Let $\delta=\frac
{1}{5}\frac{\phi^{\prime\prime}\left(  0\right)  }{\left\Vert \phi
^{\prime\prime\prime}\right\Vert _{\infty}}$, let $\psi\in C_{0}^{\infty
}\left(  -\delta,\delta\right)  $ and let%
\begin{equation}
I\left(  \lambda\right)  =\int_{\mathbb{R}}e^{i\lambda\phi\left(  x\right)
}\psi\left(  x\right)  dx\ .\label{onedash}%
\end{equation}
Then there exists $c>0$ such that%
\[
\left\vert I\left(  \lambda\right)  -\psi\left(  0\right)  \sqrt{\frac{2\pi
}{\lambda\phi^{\prime\prime}\left(  0\right)  }}e^{i\pi/4}\right\vert
\leqslant c\ \frac{1}{\lambda}\ .
\]
The constant $c$ depends on $\kappa_{1}$ and $\kappa_{2}$, where $\phi
^{\prime\prime}\left(  0\right)  \geqslant\kappa_{1}$, $\left\Vert
\phi\right\Vert _{C^{5}}\leqslant\kappa_{2}$ and $\left\Vert \psi\right\Vert
_{C^{2}}\leqslant\kappa_{2}$.
\end{proposition}

The proof is not short just because we want a constant $c$ that depends on the
norms of the functions and not on the functions themselves.

The proof of Proposition \ref{Asymptotic} needs a few lemmas.

\begin{lemma}
\label{Omega}Let $\phi\in C^{\infty}\left(  -\delta,\delta\right)  $ be a
smooth function and let, for $|x|<\delta$,
\[
\omega\left(  x\right)  =x^{-k}\int_{0}^{x}\left(  x-t\right)  ^{n}\phi\left(
t\right)  dt
\]
with $n,k\geqslant0$. Then for $0\leqslant r\leqslant n+1-k$ there exists $c$,
independent of $\phi$, such that%
\[
\left\vert \omega^{\left(  r\right)  }\left(  x\right)  \right\vert \leqslant
c\ \delta^{n+1-k-r}\left\Vert \phi\right\Vert _{\infty}\ .
\]

\end{lemma}

\begin{proof}
Clearly%
\[
\left\vert \omega\left(  x\right)  \right\vert \leqslant\left\vert
x\right\vert ^{-k}\left\vert x\right\vert ^{n+1}\left\Vert \phi\right\Vert
_{\infty}\leqslant\delta^{n+1-k}\left\Vert \phi\right\Vert _{\infty}\ .
\]
We claim that, for $1\leqslant r\leqslant n+1-k,$ the derivative $\omega
^{\left(  r\right)  }\left(  x\right)  $ is a linear combination of terms of
the form%
\[
x^{-\alpha}\int_{0}^{x}\frac{\left(  x-t\right)  ^{\beta}}{\beta!}%
\ \phi\left(  t\right)  dt
\]
with $\beta-\alpha=n-k-r$ and $\beta\geqslant0$. The proof is by induction and it is enough to observe that 
\begin{align*}
&  \frac{d}{dx}\left(  x^{-\alpha}\int_{0}^{x}\frac{\left(  x-t\right)
^{\beta}}{\beta!}\phi\left(  t\right)  dt\right)  \\
&  =-\alpha x^{-\alpha-1}\int_{0}^{x}\frac{\left(  x-t\right)  ^{\beta}}%
{\beta!}\phi\left(  t\right)  dt+\beta x^{-\alpha}\int_{0}^{x}\frac{\left(
x-t\right)  ^{\beta-1}}{\beta!}\ \phi\left(  t\right)  \ dt\ .
\end{align*}
Hence%
\begin{align*}
\left\vert \omega^{\left(  r\right)  }\left(  x\right)  \right\vert  &
\leqslant c\sum_{\alpha+\beta=n-k-r,~\beta\geqslant0,}\left\vert x\right\vert
^{\alpha}\int_{0}^{\left\vert x\right\vert }\left\vert x\right\vert ^{\beta
}\left\vert \phi\left(  t\right)  \right\vert dt\\
&  \leqslant c\sum_{\alpha+\beta=n-k-r,~\beta\geqslant0,}\delta^{\alpha
+\beta+1}\left\Vert \phi\right\Vert _{\infty}\leqslant c\delta^{n+1-k-r}%
\left\Vert \phi\right\Vert _{\infty}\ .
\end{align*}

\end{proof}

\begin{lemma}
\label{Psi}Let $\phi\in C^{\infty}\left(  -\delta,+\delta\right)  $ such that
\[
\phi\left(  0\right)  =\phi^{\prime}\left(  0\right)  =\cdots=\phi^{\left(
k-1\right)  }\left(  0\right)  =0 \;.. 
\]
Then the function%
\[
\psi\left(  x\right)  =\frac{\phi\left(  x\right)  }{x^{k}}%
\]
is smooth and for every integer $n\geqslant0$ we have%
\[
\left\Vert \psi\right\Vert _{C^{n}}\leqslant c\left\Vert \phi\right\Vert
_{C^{n+k}}\ .
\]

\end{lemma}

\begin{proof}
By the integral form of the remainder in Taylor's theorem, for every $n$ we
can write
\begin{align*}
\phi\left(  x\right)    & =\frac{\phi^{\left(  k\right)  }\left(  0\right)
}{k!}x^{k}+\cdots+\frac{\phi^{\left(  n+k-1\right)  }\left(  0\right)
}{\left(  n+k-1\right)  !}x^{n+k-1}\\
& +\int_{0}^{x}\frac{\left(  x-t-1\right)  ^{n+k-1}}{\left(  n+k-1\right)
!}\phi^{\left(  n+k\right)  }\left(  t\right)  dt
\end{align*}
(if $n=0$ then only the integral appears). Let%
\[
\omega\left(  x\right)  =x^{-k}\int_{0}^{x}\frac{\left(  x-t-1\right)
^{n+k-1}}{\left(  n+k-1\right)  !}\phi^{\left(  n+k\right)  }\left(  t\right)
dt\ .
\]
Then, by Lemma \ref{Omega} we have
\begin{align*}
\left\Vert \psi\right\Vert _{C^{n}} &  \leqslant c\left\Vert \phi\right\Vert
_{C^{n+k-1}}+\left\Vert \omega\right\Vert _{C^{n}}\\
&  \leqslant c\left\Vert \phi\right\Vert _{C^{n+k-1}}+c\left\Vert
\phi^{\left(  n+k\right)  }\right\Vert _{\infty}\leqslant c\left\Vert
\phi\right\Vert _{C^{n+k}}\ .
\end{align*}

\end{proof}

\begin{lemma}
\label{Change of variables}There exist absolute constants $c_{1},c_{2}>0$ such
that if $\phi\in C^{\infty}\left(  -\infty,+\infty\right)  $ is a convex
function satisfying $\phi\left(  0\right)  =\phi^{\prime}\left(  0\right)
=0$, $\phi^{\prime\prime}\left(  0\right)  >0$, and $\delta=\frac{\phi
^{\prime\prime}\left(  0\right)  }{\left\Vert \phi^{\prime\prime\prime
}\right\Vert _{\infty}}$, then%
\[
g\left(  x\right)  =x\sqrt{\frac{\phi\left(  x\right)  }{x^{2}}}\ .
\]
is smooth and invertible in $\left(  -\delta,\delta\right)  $. Moreover
\begin{equation}
g^{\prime}\left(  0\right)  =\sqrt{\frac{\phi^{\prime\prime}\left(  0\right)
}{2}}\label{twodashes}%
\end{equation}
and, for $\left\vert x\right\vert <\delta$,%
\[
c_{1}\sqrt{\phi^{\prime\prime}\left(  0\right)  }\leqslant g^{\prime}\left(
x\right)  \leqslant c_{2}\sqrt{\phi^{\prime\prime}\left(  0\right)  }\ .
\]
Finally $\left\Vert g\right\Vert _{C^{n}}$ can be bounded from above by a
constant that depends only on $\left\Vert \phi\right\Vert _{C^{2+n}}$, and
from below by a constant that depends only on $\phi^{\prime\prime}\left(
0\right)  $.
\end{lemma}

\begin{proof}
The integral form of the remainder in Taylor's theorem and Lemma \ref{Omega}
yield%
\[
\frac{\phi\left(  x\right)  }{x^{2}}=\frac{\phi^{\prime\prime}\left(
0\right)  }{2}+x^{-2}\int_{0}^{x}\frac{\left(  x-t\right)  ^{2}}{2}%
\phi^{\prime\prime\prime}\left(  t\right)  \ dt\ ,
\]
so that, for $\left\vert x\right\vert <\delta$, we have%
\[
\left\vert \frac{\phi\left(  x\right)  }{x^{2}}-\frac{\phi^{\prime\prime
}\left(  0\right)  }{2}\right\vert \leqslant\frac{1}{6}\delta\left\Vert
\phi^{\prime\prime\prime}\right\Vert _{\infty}=\frac{1}{6}\phi^{\prime\prime
}\left(  0\right)  \ .
\]
Hence, for $\left\vert x\right\vert <\delta$,%
\[
\frac{1}{3}\ \phi^{\prime\prime}\left(  0\right)  \leqslant\frac{\phi\left(
x\right)  }{x^{2}}\leqslant\frac{2}{3}\ \phi^{\prime\prime}\left(  0\right)
\ .
\]
Observe that this and Lemma \ref{Psi} imply that $g\left(  x\right)  $ is
smooth. Similarly%
\[
\left\vert \frac{\phi^{\prime}\left(  x\right)  }{x}-\phi^{\prime\prime
}\left(  0\right)  \right\vert   =\left\vert x^{-1}\int_{0}^{x}\left(
x-t\right)  \phi^{\prime\prime\prime}\left(  t\right)  dt\right\vert 
  \leqslant\frac{1}{2}\delta\left\Vert \phi^{\prime\prime\prime}\right\Vert
_{\infty}=\frac{1}{2}\phi^{\prime\prime}\left(  0\right)  \ ,
\]
so that%
\[
\frac{1}{2}\phi^{\prime\prime}\left(  0\right)  \leqslant\frac{\phi^{\prime
}\left(  x\right)  }{x}\leqslant\frac{3}{2}\phi^{\prime\prime}\left(
0\right)  \ .
\]
Finally, since%
\[
g^{\prime}\left(  x\right)  =\frac{1}{2}\frac{\phi^{\prime}\left(  x\right)
}{x}\left(  \frac{\phi\left(  x\right)  }{x^{2}}\right)  ^{-1/2}\ ,
\]
there are absolute constants $c_{1},c_{2}>0$ such that%
\[
c_{1}\sqrt{\phi^{\prime\prime}\left(  0\right)  }\leqslant\left\vert
g^{\prime}\left(  x\right)  \right\vert \leqslant c_{2}\sqrt{\phi
^{\prime\prime}\left(  0\right)  }\ .
\]
Observe that%
\[
\frac{d^{n}}{dx^{n}}\left(  x\sqrt{\frac{\phi\left(  x\right)  }{x^{2}}%
}\right)  \leq c^{\ast}\ ,
\]
where the constant $c^{\ast}$ depends on a lower bound for $\frac{\phi\left(
x\right)  }{x^{2}}$ and a lower bound for $\frac{d^k}{dx^{k}}\left(  \frac
{\phi\left(  x\right)  }{x^{2}}\right)  $, when $k\leq n$. Then, by Lemma
\ref{Psi}, $c^{\ast}$ depends on a lower bound of $\phi^{\prime\prime}\left(  0\right)  $ and
on $\left\Vert \phi\right\Vert _{C^{n+2}}$.
\end{proof}

\begin{proof}
[Proof of Proposition \ref{Asymptotic}]Let $I\left(  \lambda\right)  $ be as
in (\ref{onedash}). Again let $g=x\sqrt{\frac{\phi\left(  x\right)  }{x^{2}}}%
$. Then $\left[  g\left(  x\right)  \right]  ^{2}=\phi\left(  x\right)  $, so
that the change of variables $u=g\left(  x\right)  $  and Lemma
\ref{Change of variables} yield%
\[
I\left(  \lambda\right)  =\int_{\mathbb{R}}e^{i\lambda u^{2}}\frac{\psi\left(
g^{-1}\left(  u\right)  \right)  }{g^{\prime}\left(  g^{-1}\left(  u\right)
\right)  }du=\int_{\mathbb{R}}e^{i\lambda u^{2}}h\left(  u\right)  du\ ,
\]
with $h\left(  u\right)  $ smooth and compactly supported. Let $\eta\in
C_{0}^{\infty}\left(  -\infty,+\infty\right)  $ such that $\eta\left(
u\right)  \equiv1$ on the support of $h\left(  u\right)  $ and let%
\[
R\left(  u\right)  =\frac{h\left(  u\right)  e^{u^{2}}-h\left(  0\right)  }%
{u}\ .
\]
Then%
\begin{align*}
I\left(  \lambda\right)   &  =\int_{\mathbb{R}}e^{i\lambda u^{2}}h\left(
u\right)  \eta\left(  u\right)  du \\
&=\int_{\mathbb{R}}e^{i\lambda u^{2}%
}e^{-u^{2}}\left[  h\left(  u\right)  e^{u^{2}}\right]  \eta\left(  u\right)
du\\
&  =\int_{\mathbb{R}}e^{i\lambda u^{2}}e^{-u^{2}}\left[  h\left(  0\right)
+uR\left(  u\right)  \right]  \eta\left(  u\right)  du\\
&  =h\left(  0\right)  \int_{\mathbb{R}}e^{i\lambda u^{2}}e^{-u^{2}}%
\eta\left(  u\right)  du+\int_{\mathbb{R}}e^{i\lambda u^{2}}e^{-u^{2}%
}uR\left(  u\right)  \eta\left(  u\right)  du\\
&  =h\left(  0\right)  \int_{\mathbb{R}}e^{i\lambda u^{2}}e^{-u^{2}%
}du+h\left(  0\right)  \int_{\mathbb{R}}e^{i\lambda u^{2}}e^{-u^{2}}\left[
1-\eta\left(  u\right)  \right]  du\\
&  +\int_{\mathbb{R}}e^{i\lambda u^{2}}e^{-u^{2}}uR\left(  u\right)
\eta\left(  u\right)  du\\
&  =I_{1}\left(  \lambda\right)  +I_{2}\left(  \lambda\right)  +I_{3}\left(
\lambda\right)  \ .
\end{align*}
The integral in $I_{1}\left(  \lambda\right)  $ can be computed through a
familiar trick:%
\begin{align*}
\left(  \int_{-\infty}^{+\infty}e^{i\lambda u^{2}}e^{-u^{2}}du\right)  ^{2} &
=\int_{-\infty}^{+\infty}\int_{-\infty}^{+\infty}e^{i\lambda\left(
u^{2}+v^{2}\right)  }e^{-\left(  u^{2}+v^{2}\right)  }dudv\\
&  =\int_{0}^{2\pi}\int_{0}^{+\infty}e^{\left(  i\lambda-1\right)  \rho^{2}%
}\rho d\rho d\theta  =\frac{\pi}{1-i\lambda}\ .
\end{align*}
Hence (\ref{twodashes}) yields%
\[
I_{1}\left(  \lambda\right)  =h\left(  0\right)  \frac{\sqrt{\pi}}{\left(
1-i\lambda\right)  ^{1/2}}=\frac{\psi\left(  0\right)  }{\sqrt{\phi
^{\prime\prime}\left(  0\right)  }}\ \frac{\sqrt{2\pi}}{\left(  1-i\lambda
\right)  ^{1/2}}%
\]
(here we consider the branch of $z^{1/2}$ that for $z>0$ agrees with $\sqrt
{z})$. Then, for $\lambda>1$,%
\begin{align*}
I_{1}\left(  \lambda\right)   &  =\frac{\psi\left(  0\right)  }{\sqrt
{\phi^{\prime\prime}\left(  0\right)  }}\ \sqrt{2\pi}\left(  -i\lambda\left(
1+\frac{1}{-i\lambda}\right)  \right)  ^{-1/2}\\
&  =\frac{\psi\left(  0\right)  }{\sqrt{\phi^{\prime\prime}\left(  0\right)
}}\ \frac{\sqrt{2\pi}}{\sqrt{\lambda}}\ e^{i\pi/4}+\frac{\psi\left(  0\right)
}{\sqrt{\phi^{\prime\prime}\left(  0\right)  }} \ O\left(  \frac{1}{\lambda
}\right)  \ .
\end{align*}
Integration by parts in $I_{2}\left(  \lambda\right)  $ yields%
\begin{align*}
I_{2}\left(  \lambda\right)   &  =\frac{\sqrt{2}\psi\left(  0\right)  }%
{\sqrt{\phi^{\prime\prime}\left(  0\right)  }}\int_{-\infty}^{+\infty
}e^{i\lambda u^{2}}e^{-u^{2}}\left[  1-\eta\left(  u\right)  \right]  du\\
&  =\frac{\psi\left(  0\right)  }{i\lambda\sqrt{2\phi^{\prime\prime}\left(
0\right)  }}\int_{-\infty}^{+\infty}2i\lambda ue^{i\lambda u^{2}}%
\frac{e^{-u^{2}}\left[  1-\eta\left(  u\right)  \right]  }{u}du\\
&  =\frac{\psi\left(  0\right)  }{i\lambda\sqrt{2\phi^{\prime\prime}\left(
0\right)  }}\int_{-\infty}^{+\infty}e^{i\lambda u^{2}}\frac{d}{du}\left[
\frac{e^{-u^{2}}\left[  1-\eta\left(  u\right)  \right]  }{u}\right]  du\ ,
\end{align*}
so that%
\[
\left\vert I_{2}\left(  \lambda\right)  \right\vert \leqslant c\ \frac
{1}{\lambda}\ \frac{\left\vert \psi\left(  0\right)  \right\vert }{\sqrt
{\phi^{\prime\prime}\left(  0\right)  }}%
\]
(note that we can always assume that $\eta\left(  u\right)  \equiv1$ in a
given neighbourhood of the origin). Finally,%
\begin{align*}
I_{3}\left(  \lambda\right)   &  =\frac{1}{2i\lambda}\int2i\lambda
ue^{i\lambda u^{2}}e^{-u^{2}}R\left(  u\right)  \eta\left(  u\right)  \ du\\
&  =\frac{1}{2i\lambda}\int e^{i\lambda u^{2}}\frac{d}{du}\left[  e^{-u^{2}%
}R\left(  u\right)  \eta\left(  u\right)  \right]  \ du
\end{align*}
so that
\[
\left\vert I_{3}\left(  \lambda\right)  \right\vert \leqslant\frac{1}%
{2\lambda}\int\left\vert \frac{d}{du}\left[  e^{-u^{2}}R\left(  u\right)
\eta\left(  u\right)  \right]  \right\vert \ du\ .
\]
Since%
\[
h\left(  u\right)  e^{u^{2}}=h\left(  0\right)  +h^{\prime}\left(  0\right)
u+\int_{0}^{u}\left(  u-t\right)  \frac{d^{2}}{dt^{2}}\left[  e^{t^{2}%
}h\left(  t\right)  \right]  \ dt\ ,
\]
we have%
\[
R\left(  u\right)  =h^{\prime}\left(  0\right)  +\frac{1}{u}\int_{0}%
^{u}\left(  u-t\right)  \frac{d^{2}}{dt^{2}}\left[  e^{t^{2}}h\left(
t\right)  \right]  \ dt\ ,
\]%
\[
\left\vert R\left(  u\right)  \right\vert \leqslant\left\vert h^{\prime
}\left(  0\right)  \right\vert +\sup\left\vert \frac{d^{2}}{dt^{2}}\left[
e^{t^{2}}h\left(  t\right)  \right]  \right\vert \ ,
\]
where the supremum is on the support of $h\left(  t\right)  $. We also have%
\[
R^{\prime}\left(  u\right)  =\frac{1}{u^{2}}\int_{0}^{u}t\frac{d^{2}}{dt^{2}%
}\left[  e^{t^{2}}h\left(  t\right)  \right]  \ dt\ ,
\]
so that%
\[
\left\vert R^{\prime}\left(  u\right)  \right\vert \leqslant\sup\left\vert
\frac{d^{2}}{dt^{2}}\left[  e^{u^{2}}h\left(  t\right)  \right]  \right\vert
\leqslant c\left\Vert h\right\Vert _{C^{2}}.
\]
Since%
\[
h\left(  t\right)  =\frac{\psi\left(  g^{-1}\left(  t\right)  \right)
}{g^{\prime}\left(  g^{-1}\left(  t\right)  \right)  }%
\]
and, by Lemma \ref{Change of variables},
\[
g^{\prime}\left(  u\right)  \approx c_{1}\sqrt{\phi^{\prime\prime}\left(
0\right)  }\ ,
\]
we can control $\left\Vert h\right\Vert _{C^{2}}$ through an upper bound on
$\left\Vert \psi\right\Vert _{C^{2}}$ and $\left\Vert g\right\Vert _{C^{3}}$,
and a lower bound on $\phi^{\prime\prime}\left(  0\right)  $. In turns, by
Lemma \ref{Change of variables},  $\left\Vert g\right\Vert _{C^{3}}$ can be
bounded by $\left\Vert \phi\right\Vert _{C^{5}}$.
\end{proof}

Asymptotic estimates for the Fourier transform of the characteristic function of a convex body with 
smooth boundary having everywhere strictly positive curvature are well known (see \cite{Hla} and \cite{Her}). 
In the next lemma we replace the above global assumption on the curvature with a local one.
\begin{lemma}
\label{Lemma Curv Pos}Let $C$ be a strictly convex planar body with smooth
boundary but for a single point that we assume to be the origin where we only
assume $C^{2}$ regularity. Let $I$ be a small closed interval contained in
$\left(  0,\pi\right)  $. For every direction $\theta\in I$ let $\sigma
_{1}\left(  \theta\right)  $ and $\sigma_{2}\left(  \theta\right)  $ be the
two points in $\partial C$ where the tangents are perpendicular to $\Theta$.
We assume that the curvatures $K\left(  \sigma_{1}\left(  \theta\right)
\right)  $ and $K\left(  \sigma_{2}\left(  \theta\right)  \right)  $ are
positive. Then%
\begin{align*}
\widehat{\chi}_{C}\left(  \rho\Theta\right)  
&  =-\frac{1}{2\pi i}\rho^{-3/2}\left[  e^{-2\pi i\rho\Theta\cdot\sigma
_{1}\left(  \theta\right)  +\pi i/4}K^{-1/2}\left(  \sigma_{1}\left(
\theta\right)  \right)  \right.  \\
&  \left.  -e^{-2\pi i\rho\Theta\cdot\sigma_{2}\left(  \theta\right)  -\pi
i/4}K^{-1/2}\left(  \sigma_{2}\left(  \theta\right)  \right)  \right]  \\
&  +\mathcal{O}\left(  \rho^{-2}\right)  \ ,
\end{align*}
with the implicit constant in $\mathcal{O}\left(  \rho^{-2}\right)  $
depending only on\  $\inf\limits_{\theta\in I}K\left(  \sigma_{j}\left(
\theta\right)  \right)  $.
\end{lemma}

\begin{proof}
By the divergence theorem we have%
\begin{equation}
\widehat{\chi}_{C}\left(  \rho\Theta\right)  =\frac{-1}{2\pi i\rho}%
\int_{\partial C}e^{-2\pi i\rho\Theta\cdot t}\Theta\cdot\nu\left(  t\right)
\ d\mu\left(  t\right)  \ ,\nonumber
\end{equation}
where $d\mu$ is the arc-length measure on $\partial C$. Let
\[
s\mapsto\Gamma\left(  s\right)
\]
be the arc-length parametrization of $\partial C$. Then%
\[
\widehat{\chi}_{C}\left(  \rho\Theta\right)  =\frac{-1}{2\pi i\rho}\int
_{0}^{1}e^{-2\pi i\rho\Theta\cdot\Gamma\left(  s\right)  }\Theta\cdot
\nu\left(  \Gamma\left(  s\right)  \right)  ds
\]
(without loss of generality we can assume that the arc-length of $\partial C$
is $1$). Observe that in the above integral the phase $\Theta\cdot
\Gamma\left(  s\right)  $ is stationary when $\Gamma\left(  s\right)
=\sigma_{j}\left(  \theta\right)  $. Let
\[
J_{j}=\left\{  s\in\left[  0,1\right]  :\Gamma\left(  s\right)  =\sigma
_{j}\left(  \theta\right)  \text{ for some }\theta\in I\right\}
\]
and let $\varphi_{1}\left(  s\right)  $ and $\varphi_{2}\left(  s\right)  $ be
cut-off functions that take value $1$ in $J_{1}$ and $J_{2}$ respectively.
Then
\begin{align}
\widehat{\chi}_{C}\left(  \rho\Theta\right)  \label{twostars}
&  =\frac{-1}{2\pi i\rho}\int_{0}^{1}e^{-2\pi i\rho\Theta\cdot\Gamma\left(
s\right)  }\Theta\cdot\nu\left(  \Gamma\left(  s\right)  \right)  \varphi
_{1}\left(  s\right)  ds\nonumber\\
&  +\frac{-1}{2\pi i\rho}\int_{0}^{1}e^{-2\pi i\rho\Theta\cdot\Gamma\left(
s\right)  }\Theta\cdot\nu\left(  \Gamma\left(  s\right)  \right)  \varphi
_{2}\left(  s\right)  ds\nonumber\\
&  +\frac{-1}{2\pi i\rho}\int_{0}^{1}e^{-2\pi i\rho\Theta\cdot\Gamma\left(
s\right)  }\Theta\cdot\nu\left(  \Gamma\left(  s\right)  \right)  \left[
1-\varphi_{1}\left(  s\right)  -\varphi_{2}\left(  s\right)  \right]
ds\nonumber\\
&  =A_{1}+A_{2}+A_{3}\ , \nonumber
\end{align}
say. The integral in $A_{3}$ can be easily estimated since in the support of
$\left[  1-\varphi_{1}\left(  s\right)  -\varphi_{2}\left(  s\right)  \right]
$ the phase is not stationary and we can integrate by parts. Therefore we
obtain%
\[
\left\vert A_{3}\right\vert \leqslant c\rho^{-2}\ .
\]
In the integral in $A_{1}$ the phase is stationary at one point, say
$\overline{s}$ where%
\[
\Theta\cdot\Gamma^{\prime}\left(  \overline{s}\right)  =0\ .
\]
Observe that at the point $\overline{s}$ we have
\[
\Theta\cdot\Gamma^{\prime\prime}\left(  \overline{s}\right)  =\left\vert
\Gamma^{\prime\prime}\left(  \overline{s}\right)  \right\vert =K\left(
\sigma_{1}\left(  \theta\right)  \right) \;,
\]
where $K\left(  \sigma_{1}\left(  \theta\right)  \right)  $ denotes the
curvature of $\partial C$ at $\sigma_{1}\left(  \theta\right)  =\Gamma\left(
\overline{s}\right)  $. By Proposition \ref{Asymptotic} we have%
\begin{align}
A_{1} &  =-\frac{e^{-2\pi i\rho\Theta\cdot\Gamma\left(  \overline{s}\right)
}}{2\pi i\rho}\int_{0}^{1}e^{2\pi i\rho\left[  \Theta\cdot\Gamma\left(
\overline{s}\right)  -\Theta\cdot\Gamma\left(  s\right)  \right]  }\Theta
\cdot\nu\left(  \Gamma\left(  s\right)  \right)  \varphi_{1}\left(  s\right)
ds\nonumber\\
&  =-\frac{e^{-2\pi i\rho\Theta\cdot\sigma_{1}\left(  \theta\right)  }}{2\pi
i\rho}\sqrt{\frac{2\pi}{2\pi\rho K\left(  \sigma_{1}\left(  \theta\right)
\right)  }}\ e^{i\pi/4}+O\left(  \rho^{-2}\right)  \label{Stima A_1}\\
&  =-\frac{1}{2\pi i}\rho^{-3/2}e^{-2\pi i\rho\Theta\cdot\sigma_{1}\left(
\theta\right)  +i\pi/4}K^{-1/2}\left(  \sigma_{1}\left(  \theta\right)
\right)  +O\left(  \rho^{-2}\right)  \ .\nonumber
\end{align}
Similarly%
\begin{align*}
A_{2} &  =-\frac{e^{-2\pi i\rho\Theta\cdot\Gamma\left(  \overline{s}\right)
}}{2\pi i\rho}\int_{0}^{1}e^{-2\pi i\rho\left[  \Theta\cdot\Gamma\left(
s\right)  -\Theta\cdot\Gamma\left(  \overline{s}\right)  \right]  }\Theta
\cdot\nu\left(  \Gamma\left(  s\right)  \right)  \varphi_{2}\left(  s\right)
ds\\
&  =\frac{e^{-2\pi i\rho\Theta\cdot\sigma_{2}\left(  \theta\right)  }}{2\pi
i\rho}\sqrt{\frac{2\pi}{2\pi\rho K\left(  \sigma_{2}\left(  \theta\right)
\right)  }}\ e^{-i\pi/4}+O\left(  \rho^{-2}\right)  \\
&  =\frac{e^{-2\pi i\rho\Theta\cdot\sigma_{2}\left(  \theta\right)  }}{2\pi
i}\rho^{-3/2}K^{-1/2}\left(  \sigma_{2}\left(  \theta\right)  \right)
e^{-i\pi/4}+O\left(  \rho^{-2}\right)  \ .
\end{align*}

\end{proof}

We can now prove the following result (see \cite{BRT} for a different proof).

\begin{theorem}
For every $\gamma>2$ and $p\geq1$ we have%
\[
\left\{  \int_{SO\left(  2\right)  }\int_{\mathbb{T}^{2}}\left\vert
\mathcal{D}\left(  R\sigma\left(  C_{\gamma}\right)  +t\right)  \right\vert
^{p}dtd\sigma\right\}  ^{1/p}\geqslant c\ R^{1/2}\ .
\]

\end{theorem}

\begin{proof}
By our assumptions on $C_{\gamma}$ there is a positive constant $\kappa$ and
an interval $I\subset\left[  -\pi/2-\varepsilon,-\pi/2+\varepsilon\right]  $
such $K\left(  \sigma_{2}\left(  \theta\right)  \right)  >\kappa$ whenever
$\theta\notin I$. Since (on the side close to the origin) $K\left(  \sigma
_{1}\left(  \theta\right)  \right)  \rightarrow0$ as $\theta\rightarrow0$
there is an interval $J\subset I$ such that $K\left(  \sigma\left(  _{1}%
\theta\right)  \right)  <\kappa/2$ for all $\theta\in J$. Then the asymptotic
expansion in Lemma \ref{Lemma Curv Pos} yields
\begin{align}
&  \int_{0}^{2\pi}\left\vert \widehat{\chi}_{C}\left(  \rho\Theta\right)
\right\vert \ d\theta>\int_{J\cup\left(  J+\pi\right)  }\left\vert
\widehat{\chi}_{C}\left(  \rho\Theta\right)  \right\vert \ d\theta
\label{dasotto}\\
&  \geqslant c\rho^{-3/2}\int_{J\cup\left(  J+\pi\right)  }\left\vert
K^{-1/2}\left(  \sigma_{1}\left(  \theta\right)  \right)  -K^{-1/2}\left(
\sigma_{2}\left(  \theta\right)  \right)  \right\vert -c_{1}\rho^{-2}\geqslant
c\rho^{-3/2}\ .\nonumber
\end{align}
Then, for every $0\neq k\in\mathbb{Z}^{2}$, (\ref{dasotto}) and an
orthogonality argument yield%
\begin{align*}
&  \left\{  \int_{SO\left(  2\right)  }\int_{\mathbb{T}^{2}}\left\vert
\mathcal{D}\left(  R\sigma\left(  C_{\gamma}\right)  +t\right)  \right\vert
^{p}dtd\sigma\right\}  ^{1/p}\\
&  =\left\{  \int_{SO\left(  2\right)  }\left(  \left\{  \int_{\mathbb{T}^{2}%
}\left\vert \mathcal{D}\left(  R\sigma\left(  C_{\gamma}\right)  +t\right)
\right\vert ^{p}dt\right\}  ^{1/p}\right)  ^{p}d\sigma\right\}  ^{1/p}\\
&  \geqslant R^{2}\left\{  \int_{SO\left(  2\right)  }\left\vert
\int_{\mathbb{T}^{2}}\left(  \sum_{m\neq0}\widehat{\chi}_{C_{\gamma}}\left(
R\sigma\left(  m\right)  \right)  e^{2\pi im\cdot t}\right)  e^{-2\pi ik\cdot
t}dt\right\vert ^{p}d\sigma\right\}  ^{1/p}\\
&  \geqslant cR^{2}\left\{  \int_{SO\left(  2\right)  }\left\vert
\widehat{\chi}_{C_{\gamma}}\left(  R\sigma\left(  k\right)  \right)
\right\vert ^{p}d\sigma\right\}  ^{1/p}\geqslant cR^{1/2}\ .
\end{align*}

\end{proof}

The upper bound $R^{1/2}$ still holds true for suitable rotations of
$C_{\gamma}$. See \cite{BCGGT}.

\begin{theorem}
Let $\widetilde{C}_{\gamma}$ be a rotated copy of $C_{\gamma}$ and we assume
that the outward unit normal $\left(  \alpha,\beta\right)  $ at the flat point
satisfies the following Diophantine condition: for every given $\delta
<2/\left(  \gamma-2\right)  $ there exists $c>0$ such that for every positive
integer $n$ we have%
\[
\left\Vert n\frac{\alpha}{\beta}\right\Vert \geqslant\frac{c}{n^{1+\delta}}\ ,
\]
where $\left\Vert x\right\Vert $ is the distance of the real number $x$ from
the integers. Then%
\[
\left\{  \int_{\mathbb{T}^{2}}\left\vert \mathcal{D}\left(  R\widetilde
{C}_{\gamma}+t\right)  \right\vert ^{2}dt\right\}  ^{1/2}\leqslant cR^{1/2}\ .
\]

\end{theorem}

\begin{proof}
Of course we may assume $\left\vert \alpha\right\vert <\left\vert
\beta\right\vert $. We write%
\begin{align*}
&  \int_{\mathbb{T}^{2}}\left\vert \mathcal{D}\left(  R\widetilde{C}_{\gamma
}+t\right)  \right\vert ^{2}dt=R^{4}\sum_{\left(  m_{1},m_{2}\right)
\neq\left(  0,0\right)  }\left\vert \widehat{\chi}_{\widetilde{C}_{\gamma}%
}\left(  Rm_{1},Rm_{2}\right)  \right\vert ^{2}\\
&  \leqslant R^{4}\sum_{0<\left\vert -\beta m_{1}+\alpha m_{2}\right\vert
<1/2}\left\vert \widehat{\chi}_{\widetilde{C}_{\gamma}}\left(  Rm_{1}%
,Rm_{2}\right)  \right\vert ^{2}\\
&  +R^{4}\sum_{1/2\leqslant\left\vert -\beta m_{1}+\alpha m_{2}\right\vert
<\left\vert \alpha m_{1}+\beta m_{2}\right\vert }\left\vert \widehat{\chi
}_{\widetilde{C}_{\gamma}}\left(  Rm_{1},Rm_{2}\right)  \right\vert ^{2}\\
&  +R^{4}\sum_{0<\left\vert \alpha m_{1}+\beta m_{2}\right\vert \leqslant
\left\vert -\beta m_{1}+\alpha m_{2}\right\vert }\left\vert \widehat{\chi
}_{\widetilde{C}_{\gamma}}\left(  Rm_{1},Rm_{2}\right)  \right\vert ^{2}\\
&  =A+B+C\ ,
\end{align*}
say. We are going to apply the estimates in Theorem \ref{decadCgamma}, with
\[
\psi\approx\frac{\left\vert -\beta m_{1}+\alpha m_{2}\right\vert }{\sqrt
{m_{1}^{2}+m_{2}^{2}}}\ .
\]
In order to estimate $A$ we observe that $0<\left\vert -\beta m_{1}+\alpha
m_{2}\right\vert <1/2$ implies $m_{1}^{2}+m_{2}^{2}\approx m_{2}^{2}$ and
therefore
\begin{align*}
A &  \leqslant cR\sum_{0<\left\vert -\beta m_{1}+\alpha m_{2}\right\vert
<1/2}\psi^{-\left(  \gamma-2\right)  /\left(  \gamma-1\right)  }\left(
m_{1}^{2}+m_{2}^{2}\right)  ^{-3}\\
&  \leqslant cR\sum_{0<\left\vert -\beta m_{1}+\alpha m_{2}\right\vert
<1/2}\left\vert -\beta m_{1}+\alpha m_{2}\right\vert ^{-\left(  \gamma
-2\right)  /\left(  \gamma-1\right)  }\left\vert m_{2}\right\vert
^{-2-1/\left(  \gamma-1\right)  }\\
&  \leqslant cR\sum_{0<\left\vert -\beta m_{1}+\alpha m_{2}\right\vert
<1/2}\left\Vert m_{2}\frac{\alpha}{\beta}\right\Vert ^{-\left(  \gamma
-2\right)  /\left(  \gamma-1\right)  }\left\vert m_{2}\right\vert
^{-2-1/\left(  \gamma-1\right)  }\\
&  \leqslant cR\sum_{0<\left\vert -\beta m_{1}+\alpha m_{2}\right\vert
<1/2}\left\vert m_{2}\right\vert ^{\left(  1+\delta\right)  \left(
\gamma-2\right)  /\left(  \gamma-1\right)  }\left\vert m_{2}\right\vert
^{-2-1/\left(  \gamma-1\right)  }=cR\ ,
\end{align*}
because $\delta<2/\left(  \gamma-2\right)  $. As for $B$ we can replace the
sum with an integral and have
\begin{align*}
B &  \leqslant cR\sum_{1/2\leqslant\left\vert -\beta m_{1}+\alpha
m_{2}\right\vert <\left\vert \alpha m_{1}+\beta m_{2}\right\vert }\left\vert
-\beta m_{1}+\alpha m_{2}\right\vert ^{-\left(  \gamma-2\right)  /\left(
\gamma-1\right)  }\\
&  \times\left\vert \alpha m_{1}+\beta m_{2}\right\vert ^{-2-1/\left(
\gamma-1\right)  }\\
&  \leqslant cR\int_{1/2\leqslant\left\vert \xi\right\vert \leqslant\left\vert
s\right\vert }\left\vert \xi\right\vert ^{-\left(  \gamma-2\right)  /\left(
\gamma-1\right)  }\left\vert s\right\vert ^{-2-1/\left(  \gamma-1\right)
}\ d\xi ds \\
&\leqslant cR\ .
\end{align*}

Finally%
\begin{align*}
C &  \leqslant R\sum_{0<\left\vert \alpha m_{1}+\beta m_{2}\right\vert
\leqslant\left\vert -\beta m_{1}+\alpha m_{2}\right\vert }\left\vert \left(
m_{1},m_{2}\right)  \right\vert ^{-3}\\
&  \leqslant cR\sum_{\left(  m_{1},m_{2}\right)  \neq\left(  0,0\right)
}\left\vert \left(  m_{1},m_{2}\right)  \right\vert ^{-3}=cR\ .
\end{align*}

\end{proof}

\begin{remark}We recall that if $\omega$ is an irrational algebraic number, then Roth's
theorem \cite{Rot0} says that for every $\varepsilon>0$ there exists $c>0$
such that%
\[
\left\Vert n\omega\right\Vert \geq\frac{1}{n^{1+\varepsilon}}\ .
\]
\end{remark}

\section{Irregularities of distribution for $C_{\gamma}$}

The above upper bound $R^{1/2}$ for the discrepancy is best possible in the
following sense. Let the integer $N$ be a square\footnote{Actually it is not
necessary to choose $N$ to be a square, see \cite[p. 3533]{BIT}}, say $N=M^{2}$.
Then the set%
\[
\frac{1}{M}\mathbb{Z}^{2}\cap\left[  -\frac{1}{2},\frac{1}{2}\right)  ^{2}%
\]
contains $N$ points and, for a convex planar body $C\subset\left[  -\frac
{1}{2},\frac{1}{2}\right)  ^{2}$, we have%
\[
\mathrm{card}\left(  \mathbb{Z}^{2}\cap MC\right)  =\mathrm{card}\left(
\frac{1}{M}\mathbb{Z}^{2}\cap C\right)  \ .
\]
Then the study of integer points in large convex bodies is a counterpart to a
classical \textquotedblleft irregularities of distribution\textquotedblright%
\ problem (see \cite{BC},\cite{Mat}). In other words, it is a particular answer to
 the problem of choosing $N$ points in $\left[  -1/2,1/2\right)  ^{2}$ to approximate the area of a given family of sets.

We have the following result.

\begin{theorem}
\label{irredistr}Let $C_{\gamma}$ be as in the Introduction. Let $N$ be a
positive large integer. Then there exists a constant $c>0$ such that for every
finite set 
\[
\left\{  u\left(  j\right)  \right\}  _{j=1}^{N}\subset\left[
-1/2,1/2\right)  ^{2} 
\]
we have%
\begin{equation}
\left\{  \int_{1/2}^{1}\int_{\mathbb{T}^{2}}\left\vert -N\left\vert C_{\gamma
}\right\vert +\sum_{j=1}^{N}\chi_{\tau C_{\gamma}}\left(  u\left(  j\right)
+t\right)  \right\vert ^{2}\ dtd\tau\right\}  ^{1/2}\geqslant c\ N^{1/4}%
\ .\label{nunquarto}%
\end{equation}

\end{theorem}

\begin{corollary}
Let $C_{\gamma}$ and $N$ be as in the previous theorem. Then there exists a
dilated and translated copy $\widetilde{C}_{\gamma}$ of $C_{\gamma}$ such that%
\[
\left\vert -N\left\vert C_{\gamma
}\right\vert +\sum_{j=1}^{N}\chi_{\widetilde{C}_{\gamma}}\left(  u\left(  j\right)
\right)  \right\vert \geqslant c\ N^{1/4}%
\ .
\]

\end{corollary}

Note that in order to compare (\ref{nunquarto}) with the results in the
previous section, we should take $R=N^{1/2}$.

To prove Theorem \ref{irredistr} we first need a mild variant of a classical
result due to J.W.S. Cassels (see e.g. \cite{Mon}). For every positive real
number $K$ let we consider the square
\[
Q_{K}=\left\{  m=(m_{1},m_{2})\in\mathbb{Z}^{2}:\ \left\vert m_{1}\right\vert
\leqslant K\;,\;\left\vert m_{2}\right\vert \leqslant K\right\}  \ .
\]

\begin{lemma}
\label{lemmaturan}For every choice of positive integers $H,N$ and $L$, such
that $H<\sqrt{L}$, let
\begin{equation}
\widetilde{Q}_{N}=Q_{\sqrt{LN}}\diagdown Q_{H}\;.\label{quadrocasselakdj}%
\end{equation}
Then for every finite set $\left\{  u(j)\right\}  _{j=1}^{N}\subset
\mathbb{T}^{2}$ we have
\begin{equation}
\sum_{0\neq m\in\widetilde{Q}_{N}}\left\vert \sum_{j=1}^{N}e^{2\pi im\cdot
u(j)}\right\vert ^{2}\geqslant\left(  L-H^{2}\right)  N^{2}%
\;.\label{casselstoprovehdahdh}%
\end{equation}

\end{lemma}

\begin{proof}
Since%
\[
\sum_{\left\vert m_{1}\right\vert \leqslant H}\sum_{\left\vert m_{2}%
\right\vert \leqslant H}\left\vert \sum_{j=1}^{N}e^{2\pi im\cdot
u(j)}\right\vert ^{2}\leqslant N^{2}H^{2}%
\]
it is enough to show that%
\[
\sum_{\left\vert m_{1}\right\vert \leqslant\sqrt{LN}}\sum_{\left\vert
m_{2}\right\vert \leqslant\sqrt{LN}}\left\vert \sum_{j=1}^{N}e^{2\pi im\cdot
u(j)}\right\vert ^{2}\geqslant LN^{2}\;,
\]
and this will follow from the inequality%
\begin{equation}
\sum_{\left\vert m_{1}\right\vert \leqslant\left[  \sqrt{LN}\right]  }%
\sum_{\left\vert m_{2}\right\vert \leqslant\left[  \sqrt{LN}\right]
}\left\vert \sum_{j=1}^{N}e^{2\pi im\cdot u(j)}\right\vert ^{2}\geqslant
N\left(  \left[  \sqrt{LN}\right]  +1\right)  ^{2}\;.\label{prontoturannanana}%
\end{equation}
Indeed let $u\left(  \ell\right)  =\left(  u_{1}\left(  \ell\right)  ,u_{2}\left(
\ell\right)  \right)  $. Then the LHS of (\ref{prontoturannanana}) is larger
than%
\begin{align}
&  \sum_{\left\vert m_{1}\right\vert \leqslant\left[  \sqrt{LN}\right]  }%
\sum_{\left\vert m_{2}\right\vert \leqslant\left[  \sqrt{LN}\right]  }\left(
1-\frac{\left\vert m_{1}\right\vert }{\left[  \sqrt{LN}\right]  +1}\right)
\nonumber\\
&  \times\left(  1-\frac{\left\vert m_{2}\right\vert }{\left[  \sqrt
{LN}\right]  +1}\right)  \left\vert \sum_{j=1}^{N}e^{2\pi im\cdot
u(j)}\right\vert ^{L}\\
&  =\sum_{\left\vert m_{1}\right\vert \leqslant\left[  \sqrt{LN}\right]  }%
\sum_{\left\vert m_{2}\right\vert \leqslant\left[  \sqrt{LN}\right]  }\left(
1-\frac{\left\vert m_{1}\right\vert }{\left[  \sqrt{LN}\right]  +1}\right)
\left(  1-\frac{\left\vert m_{2}\right\vert }{\left[  \sqrt{LN}\right]
+1}\right)  \nonumber\\
&  \times\sum_{j=1}^{N}\sum_{k=1}^{N}e^{2\pi im\cdot\left(  u(j)-u(k)\right)
}\nonumber\\
&  =\sum_{j=1}^{N}\sum_{k=1}^{N}\sum_{\left\vert m_{1}\right\vert
\leqslant\left[  \sqrt{LN}\right]  }\left(  1-\frac{\left\vert m_{1}%
\right\vert }{\left[  \sqrt{LN}\right]  +1}\right)  e^{2\pi im_{1}\left(
u_{1}(j)-u_{1}(k)\right)  }\nonumber\\
&  \times\sum_{\left\vert m_{2}\right\vert \leqslant\left[  \sqrt{LN}\right]
}\left(  1-\frac{\left\vert m_{2}\right\vert }{\left[  \sqrt{LN}\right]
+1}\right)  e^{2\pi im_{2}\left(  u_{2}(j)-u_{2}(k)\right)  }\nonumber\\
&  =\sum_{j=1}^{N}\sum_{k=1}^{N}K_{\left[  \sqrt{LN}\right]  }\left(
u_{1}(j)-u_{1}(k)\right)  K_{\left[  \sqrt{LN}\right]  }\left(  u_{2}%
(j)-u_{2}(k)\right)  \;,\label{turannnn}%
\end{align}
where
\[
K_{M}(x)=\sum_{j=-M}^{M}\left(  1-\frac{\left\vert j\right\vert }{M+1}\right)
e^{2\pi ijx}=\frac{1}{M+1}\left(  \frac{\sin\left(  \pi\left(  M+1\right)
x\right)  }{\sin\left(  \pi x\right)  }\right)  ^{2}%
\]
is the Fej\'{e}r kernel on $\mathbb{T}$. Since $K_{M}\left(  x\right)
\geqslant0$ for every $x$, the last term in (\ref{turannnn}) is not smaller
than the \textquotedblleft diagonal\textquotedblright\
\begin{align*}
\sum_{j=1}^{N} &  K_{\left[  \sqrt{LN}\right]  }\left(  u_{1}(j)-u_{1}%
(j)\right)  K_{\left[  \sqrt{LN}\right]  }\left(  u_{2}(j)-u_{2}(j)\right)  \\
&  =N\,K_{\left[  \sqrt{LN}\right]  }(0)\,K_{\left[  \sqrt{LN}\right]
}(0)=N\left(  \left[  \sqrt{LN}\right]  +1\right)  ^{2}\;.
\end{align*}

\end{proof}

Now we need an estimate from below of $\int_{1/2}^{1}\left\vert \widehat{\chi
}_{sC_{\gamma}}\left(  k\right)  \right\vert ^{2}\ ds$, for $0\neq
k\in\mathbb{Z}^{2}$.

\begin{lemma}
\label{carattere dal basso}Let $C_{\gamma}$ be as in the Introduction. Then
there exist constants $c_{1},c_{2}>0$ such that for $\left\vert \xi\right\vert
\geqslant c_{1}$ we have%
\[
\left\{  \int_{1/2}^{1}\left\vert \widehat{\chi}_{C_{\gamma}}\left(  \tau
\xi\right)  \right\vert ^{2}\ d\tau\right\}  ^{1/2}\geqslant c_{2}\left\vert
\xi\right\vert ^{-3/2}\ .
\]

\end{lemma}

\begin{proof}
Let $\xi=\rho\Theta$, arguing as in the proof of Lemma \ref{Lemma Curv Pos} we
write%
\begin{align*}
\widehat{\chi}_{C_{\gamma}}\left(  \tau\rho\Theta\right) & =\frac{-1}{2\pi i\tau\rho}\int_{0}^{1}e^{-2\pi i\tau\rho\Theta\cdot
\Gamma\left(  s\right)  }\Theta\cdot\nu\left(  \Gamma\left(  s\right)
\right)  \varphi_{1}\left(  s\right)  ds\\
&  +\frac{-1}{2\pi i\tau\rho}\int_{0}^{1}e^{-2\pi i\tau\rho\Theta\cdot
\Gamma\left(  s\right)  }\Theta\cdot\nu\left(  \Gamma\left(  s\right)
\right)  \varphi_{2}\left(  s\right)  ds\\
&  +\frac{-1}{2\pi i\tau\rho}\int_{0}^{1}e^{-2\pi i\tau\rho\Theta\cdot
\Gamma\left(  s\right)  }\Theta\cdot\nu\left(  \Gamma\left(  s\right)
\right)  \left[  1-\varphi_{1}\left(  s\right)  -\varphi_{2}\left(  s\right)
\right]  ds\\
&  =A_{1}\left(  \tau\rho\right)  +A_{2}\left(  \tau\rho\right)  +A_{3}\left(
\tau\rho\right)  \ .
\end{align*}
We have%
\begin{align*}
&  \left\{  \int_{1/2}^{1}\left\vert \widehat{\chi}_{C_{\gamma}}\left(
\tau\xi\right)  \right\vert ^{2}\ d\tau\right\}  ^{1/2}\\
&  \geqslant\left\{  \int_{1/2}^{1}\left\vert A_{1}\left(  \tau\rho\right)
+A_{2}\left(  \tau\rho\right)  \right\vert ^{2}\ d\tau\right\}  ^{1/2}%
-\left\{  \int_{1/2}^{1}\left\vert A_{3}\left(  \tau\rho\right)  \right\vert
^{2}\ d\tau\right\}  ^{1/2}\ .
\end{align*}
Since (in $A_{3}$) in the support of $\left[  1-\varphi_{1}\left(  s\right)
-\varphi_{2}\left(  s\right)  \right]  $ the phase is not stationary,
integration by parts yields%
\[
\left\vert A_{3}\left(  \tau\rho\right)  \right\vert \leqslant c\tau^{-2}%
\rho^{-2} \;,
\]
and therefore%
\[
\left\{  \int_{1/2}^{1}\left\vert \widehat{\chi}_{C_{\gamma}}\left(
\tau\xi\right)  \right\vert ^{2}\ d\tau\right\}  ^{1/2}
  \geqslant\left\{  \int_{1/2}^{1}\left\vert A_{1}\left(  \tau\rho\right)
+A_{2}\left(  \tau\rho\right)  \right\vert ^{2}\ d\tau\right\}  ^{1/2}%
-c\rho^{-2}\ .
\]

By our assumptions on $C_{\gamma}$ we know that at least one (say the first
one) of the two integrals in $A_{1}$ and $A_{2}$ corresponds to a part of
$\partial C_{\gamma}$ where the curvature is bounded away from zero. Let
$\eta\in C_{0}^{\infty}\left(  1/2,1\right)  $ be a cut-off function such that
$0\leqslant\eta\left(  \tau\right)  \leqslant1$ and $\eta\left(  \tau\right)
\equiv1$ for $5/8\leqslant\tau\leqslant7/8$. Then%
\begin{align*}
&  \int_{1/2}^{1}\left\vert A_{1}\left(  \tau\rho\right)  +A_{2}\left(
\tau\rho\right)  \right\vert ^{2}d\tau\geqslant\int_{1/2}^{1}\left\vert
A_{1}\left(  \tau\rho\right)  +A_{2}\left(  \tau\rho\right)  \right\vert
^{2}\eta\left(  \tau\right)  d\tau\\
&  =\int_{1/2}^{1}\left(  \left\vert A_{1}\left(  \tau\rho\right)  \right\vert
^{2}+\left\vert A_{2}\left(  \tau\rho\right)  \right\vert ^{2}%
+2\operatorname{Re}\left(  A_{1}\left(  \tau\rho\right)  \overline
{A_{2}\left(  \tau\rho\right)  }\right)  \right)  \eta\left(  \tau\right)
d\tau\\
&  \geqslant\int_{1/2}^{1}\left\vert A_{1}\left(  \tau\rho\right)  \right\vert
^{2}\eta\left(  \tau\right)  d\tau+2\operatorname{Re}\int_{1/2}^{1}\left(
A_{1}\left(  \tau\rho\right)  \overline{A_{2}\left(  \tau\rho\right)
}\right)  \eta\left(  \tau\right)  d\tau
\end{align*}

For the second integral we have%
\begin{align*}
&  \int_{1/2}^{1}A_{1}\left(  \tau\rho\right)  \overline{A_{2}\left(  \tau
\rho\right)  }\eta\left(  \tau\right)  d\tau\\
&  =\frac{-1}{4\pi^{2}\rho^{2}}\int_{1/2}^{1}\tau^{-2}\int_{0}^{1}\int_{0}%
^{1}e^{2\pi i\tau\rho\Theta\cdot\left[  \Gamma\left(  w\right)  -\Gamma\left(
s\right)  \right]  }\left[  \Theta\cdot\nu\left(  \Gamma\left(  s\right)
\right)  \Theta\cdot\nu\left(  \Gamma\left(  w\right)  \right)  \right]
\\
&  ~~~~~~~~~~~\times\varphi_{1}\left(  s\right)  \varphi_{2}\left(  w\right)
dsdw\ \eta\left(  \tau\right)  d\tau\\
&  \frac{-1}{4\pi^{2}\rho^{2}}\int_{0}^{1}\int_{0}^{1}\int_{1/2}^{1}e^{2\pi
i\tau\rho\Theta\cdot\left[  \Gamma\left(  w\right)  -\Gamma\left(  s\right)
\right]  }\frac{\eta\left(  \tau\right)  }{\tau^{2}}d\tau \\
&  ~~~~~~~~~~~\times\left[  \Theta\cdot\nu\left(  \Gamma\left(  s\right)
\right)  \Theta\cdot\nu\left(  \Gamma\left(  w\right)  \right)  \right]
\varphi_{1}\left(  s\right)  \varphi_{2}\left(  w\right)  dsdw
\end{align*}
Observe that if $\ell\left(  \tau\right)  =\eta\left(  \tau\right)  /\tau^{2}%
$, then%
\[
\int_{1/2}^{1}e^{2\pi i\tau\rho\Theta\left[  \Gamma\left(  w\right)
-\Gamma\left(  s\right)  \right]  }\frac{\eta\left(  \tau\right)  }{\tau^{2}%
}\ d\tau=\widehat{\ell}\left(  \rho\Theta\left[  \Gamma\left(  w\right)
-\Gamma\left(  s\right)  \right]  \right) \;.
\]
Since $\left\vert \Theta\cdot\left[  \Gamma\left(  w\right)  -\Gamma\left(
s\right)  \right]  \right\vert \geqslant c>0$ for every $w,s$ in the supports
of $\varphi_{1}$ and $\varphi_{2}$ respectively, integration by parts gives%
\[
\int_{1/2}^{1}e^{2\pi i\tau\rho\Theta\cdot\left[  \Gamma\left(  w\right)
-\Gamma\left(  s\right)  \right]  }\frac{\eta\left(  \tau\right)  }{\tau^{2}%
}d\tau=O\left(  \rho^{-L}\right)
\]
for every $L$. It follows that%
\[
\left\{  \int_{1/2}^{1}\left\vert A_{1}\left(  \tau\rho\right)  +A_{2}\left(
\tau\rho\right)  \right\vert ^{2}d\tau\right\}  ^{1/2} 
\geqslant c\left\{
\int_{1/2}^{1}\left\vert A_{1}\left(  \tau\rho\right)  \right\vert ^{2}%
\eta\left(  \tau\right)  d\tau\right\}  ^{1/2}+O\left(  \rho^{-L}\right)  .
\]
Also, by our choice of $A_{1}$, we have%
\[
A_{1}\left(  \tau\rho\right)  =-\frac{1}{2\pi i}\left(  \tau\rho\right)
^{-3/2}e^{-2\pi i\tau\rho\Theta\cdot\sigma_{1}\left(  \theta\right)
+i\frac{\pi}{4}}K^{-1/2}\left(  \sigma_{1}\left(  \theta\right)  \right)
+O\left(  \tau^{-2}\rho^{-2}\right)
\]
so that%
\[
\left\{  \int_{1/2}^{1}\left\vert A_{1}\left(  \tau\rho\right)  \right\vert
^{2}\eta\left(  \tau\right)  d\tau\right\}  ^{1/2}\geqslant c_{1}\rho
^{-3/2}K^{-1/2}\left(  \sigma_{1}\left(  \theta\right)  \right)  -c_{2}%
\rho^{-2}\ .
\]
Finally,%
\[
\left\{  \int_{1/2}^{1}\left\vert \widehat{\chi}_{C_{\gamma}}\left(  \tau
\xi\right)  \right\vert ^{2}\ d\tau\right\}  ^{1/2}\geqslant c_{1}\rho
^{-3/2}-c_{2}\rho^{-2}\geqslant c_{3}\rho^{-3/2}%
\]
for $\rho$ large enough.
\end{proof}

\bigskip

\begin{proof}
[Proof of Theorem \ref{irredistr}]We apply Parseval theorem,
(\ref{casselstoprovehdahdh}), and Lemma \ref{carattere dal basso}, where we
choose $H=c_{1}$. Then, for $\widetilde{Q}_{N}$ as in (\ref{quadrocasselakdj}%
), we have%
\begin{align*}
&  \int_{1/2}^{1}\int_{\mathbb{T}^{2}}\left\vert -N\left\vert C_{\gamma
}\right\vert +\sum_{j=1}^{N}\chi_{\tau C_{\gamma}}\left(  u\left(  j\right)
+t\right)  \right\vert ^{2}\ dtd\tau\\
&  =\int_{1/2}^{1}\sum_{m\neq0}\left\vert \sum_{j=1}^{N}e^{2\pi im\cdot
u\left(  j\right)  }\right\vert ^{2}\left\vert \widehat{\chi}_{\tau C_{\gamma
}}\left(  m\right)  \right\vert ^{2}\ d\tau\\
&  \geqslant\sum_{m\in Q_{N}}\left\vert \sum_{j=1}^{N}e^{2\pi im\cdot u\left(
j\right)  }\right\vert ^{2}\int_{1/2}^{1}\tau^{2}\left\vert \widehat{\chi
}_{C_{\gamma}}\left(  \tau m\right)  \right\vert ^{2}\ d\tau\\
&  \geqslant c\left\vert \sqrt{N}\right\vert ^{-3}\sum_{m\in Q_{N}}\left\vert
\sum_{j=1}^{N}e^{2\pi im\cdot u\left(  j\right)  }\right\vert ^{2}\geqslant
c\ N^{1/2}\ .
\end{align*}

\end{proof}

\begin{remark}
We have already pointed out that the discrepancy
results for $C_{\gamma}$ are \textquotedblleft intermediate\textquotedblright%
\ between the case of a convex body with smooth boundary having everywhere
positive curvature, and the case of a polygon (just send $\gamma\rightarrow2$
or $\gamma\rightarrow+\infty$). This is not the case for the main result in
this section. Indeed we know that for a polygon we have a logarithmic lower
bound (see \cite{Mon}) which has a counterpart in Davenport's paper
\cite{Dav}. The \textquotedblleft explanation\textquotedblright\ is that a
polygon does not have points on the boundary with positive curvature, while
for every $\gamma<+\infty$ the convex body $C_{\gamma}$ has such points.
\end{remark}

\section{Remarks on higher dimensional cases}

Kendall's upper bound works in higher dimensions as well, Indeed let
$B=\left\{  t\in\mathbb{R}^{d}:\left\vert t\right\vert \leqslant1\right\}  $
and let $t\in\mathbb{T}^{d}=\mathbb{R}^{d}/\mathbb{Z}^{d}$. Let
\[
D_{R}\left(  \sigma,t\right)  =-R^{d}\left\vert B\right\vert +\mathrm{card}%
\left(  \left(  \sigma\left(  RB\right)  +t\right)  \cap\mathbb{Z}^{d}\right)
\ .
\]
Then, see e.g. \cite{BGT},%
\[
\left\{  \int_{\mathbb{T}^{d}}\left\vert D_{R}\left(  \sigma,t\right)
\right\vert ^{2}\ dt\right\}  ^{1/2}\leqslant c\ R^{\left(  d-1\right)  /2}\ .
\]
Interestingly (see \cite{PS}) its converse%
\[
\left\{  \int_{\mathbb{T}^{d}}\left\vert D_{R}\left(  \sigma,t\right)
\right\vert ^{2}\ dt\right\}  ^{1/2}\geqslant c_{1}\ R^{\left(  d-1\right)
/2}%
\]
holds if and only if $d\not \equiv 1\left(  \mathrm{mod\ }4\right)  $.

\bigskip

Theorem \ref{chord} does not extend to the case $d\geqslant3$. Indeed,
consider the cube $Q$ in the following figure and the Fourier transform
$\widehat{\chi}_{Q}\left(  \xi\right)  $ in the direction of $\xi$. Then
$\left\vert \widehat{\chi}_{Q}\left(  \xi\right)  \right\vert $ cannot be
controlled by the area of the triangle (i.e. the section) perpendicular to $\xi$ (at distance
$1/\left\vert \xi\right\vert $).
\[
\includegraphics[width=5cm]{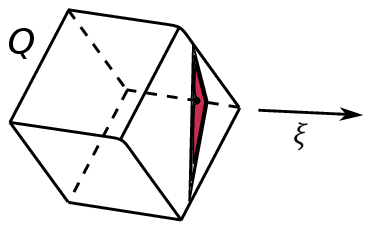}
\]
Indeed the area of the triangle decays of order $2$, so that the
\textquotedblleft parallel section function\textquotedblright\ $\ \mathbb{R}%
\ni x\mapsto h\left(  x\right)  $, which measures the areas of the sections of
$C$ perpendicular to $\xi$, has a shape similar to the following one:%
\[
\includegraphics[width=7.5cm]{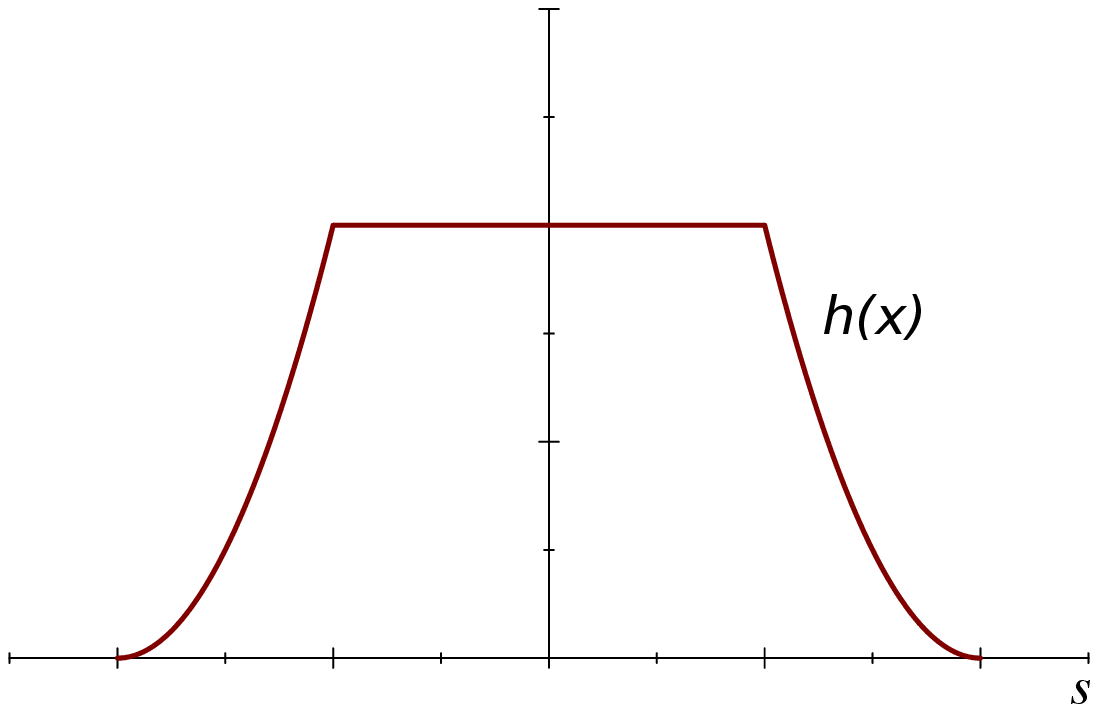}
\]
The above figure shows that the parallel section function $h\left(  x\right)
$ is more regular at the boundary of its support than inside it. Since the
Fourier transform is mostly affected by the \textquotedblleft
irregular\textquotedblright\ points, the decay of $\widehat{\chi}_{Q}\left(
\xi\right)  $ cannot be controlled by a geometric estimate around the boundary
of $Q$. Anyhow this may not be an obstacle. Indeed in the case of a ball $B$
or in the case of a convex body $C$ with smooth boundary having positive
curvature we can still use the asymptotics of Bessel functions (or more
refined estimates introduced by E. Hlawka and C. Herz) to estimate$\ \widehat
{\chi}_{C}(\xi)$. In the case of a polyhedron we may obtain fairly precise
estimates working by induction on its faces. See also \cite{BNW},\cite{BMVW} for general results concerning convexity and geometric estimates of Fourier transforms.

\bigskip

The dyadic argument in the second proof of Theorem \ref{decadCgamma} holds
true in several variables as well (see \cite{BGreenT}).

\bigskip

Theorems \ref{minore dpiu1} and \ref{maggiore dpiu1} can be extended to several
variables with the following more general assumption on $\partial C_{\gamma}$.

\begin{definition}
Let $U$ be a bounded open neighborhood of the origin in $\mathbb{R}^{d-1}$,
let $\Phi\in C^{\infty}\left(  U\setminus\left\{  0\right\}  \right)  $ and
let $\gamma>1$. For every $x\in U\setminus\left\{  0\right\}  $ let $\mu
_{1}\left(  x\right)  ,\ldots,\mu_{d-1}\left(  x\right)  $ be the eigenvalues
of the Hessian matrix of $\Phi$. We say that $\Phi\in S_{\gamma}\left(
U\right)  $ if for $j=1,\ldots,d-1,$%
\[
0<\inf_{x\in U\setminus\left\{  0\right\}  }\left\vert x\right\vert
^{2-\gamma}\mu_{j}\left(  x\right)
\]
and, for every multi-index $\alpha,$%
\[
\sup_{x\in U\setminus\left\{  0\right\}  }\left\vert x\right\vert ^{\left\vert
\alpha\right\vert -\gamma}\left\vert \frac{\partial^{\left\vert \alpha
\right\vert }\Phi}{\partial x^{\alpha}}\left(  x\right)  \right\vert <+\infty.
\]
Let $B$ be a convex body in $\mathbb{R}^{d}$, let $t\in\partial B$
and let $\gamma>2$. We say that $t$ is an isolated flat point of
order $\gamma$ if, in a neighbourhood of $t$ and in a suitable
Cartesian coordinate system with the origin in $t$, $\partial B$ is
the graph of a function $\Phi\in S_{\gamma}\left(  U\right)  $.
\end{definition}

\bigskip

Also Theorem \ref{Gari} can be extended to several variables, see \cite{Gar}.

\end{document}